\documentclass[12pt]{article}

\usepackage{amsfonts, amsmath, amsthm, amssymb}

\usepackage{indentfirst}
\usepackage[T2A]{fontenc}
\usepackage[cp1251]{inputenc}
\usepackage[english]{babel}

\usepackage{xcolor,graphics}

\usepackage[unicode]{hyperref}

\newcommand{\mes}{\operatorname{mes}}
\newcommand{\bigzero}{\text{\rm\Large0}}
\newcommand{\diag}{\operatorname{diag}}

\newtheorem{theorem}{Theorem}
\newtheorem{theorem*}{Theorem}
\newtheorem{lemma}{Lemma}
\newtheorem{corollary}{Corollary}

\title{Distribution of real algebraic integers}
\author{Denis V. Koleda}
\date{}

\begin{document}
\maketitle

\begin{abstract}
In the paper, we study the asymptotic distribution of real algebraic integers of fixed degree as their na\"{\i}ve height tends to infinity. For an arbitrary interval $I \subset \mathbb{R}$ and sufficiently large $Q>0$, we obtain an asymptotic formula for the number of algebraic integers $\alpha\in I$ of fixed degree $n$ and na\"{\i}ve height $H(\alpha)\le Q$.
In particular, we show that the real algebraic integers of degree $n$, with their height growing, tend to be distributed like the real algebraic numbers of degree $n-1$.
However, we reveal two symmetric ``plateaux'', where the distribution of real algebraic integers statistically resembles the rational integers.
\end{abstract}

\section{Introduction and main results}

\subsection{Basic definitions}

Let $p(x) = a_n x^n + \ldots + a_1 x + a_0$ be a polynomial of degree $n$,
and let $H(p)$ be its (\emph{na\"{\i}ve}, or \emph{usual}) \emph{height} defined as
\begin{equation}\label{eq-n-height-def}
H(p) = \max_{0\le i\le n} |a_i|,
\end{equation}

Let $\alpha\in\mathbb{C}$ be an algebraic number.
We define the \emph{minimal polynomial} of $\alpha$ as a nonzero polynomial~$p$ of the minimal degree $\deg(p)$ with integer coprime coefficients and positive leading coefficient such that $p(\alpha)=0$.

For the algebraic number $\alpha$, its \emph{degree} $\deg(\alpha)$ and \emph{height} $H(\alpha)$ are defined as the degree and height of the corresponding minimal polynomial.

An algebraic number is called an \emph{algebraic integer} if its minimal polynomial is monic, that is, has the leading coefficient 1.

Distinct algebraic numbers $\alpha_1$ and $\alpha_2$ are called \emph{conjugate} if they have the same minimal polynomial.
Obviously, any algebraic number $\alpha$ of degree $n$ has $n-1$ distinct conjugates (different from $\alpha$).

A real algebraic integer $\alpha$ is called a \emph{Perron number} if all its conjugates are less than~$\alpha$ in absolute value.

We denote by $\# S$ the number of elements in a finite set $S$;
$\mes_{k} S$ denotes the $k$--dimensional Lebesgue measure of a set $S \subset \mathbb{R}^{d}$ ($k \le d$).
The length of an interval $I$ is denoted by $|I|$.
The Euclidean norm of a vector $\mathbf{x}\in\mathbb{R}^k$ is denoted by $\|\mathbf{x}\|$.
To denote asymptotic relations between functions, we use the Vinogradov symbol $\ll$:
expression $f \ll g$ denotes that $f \le c\,g$, where $c$ is a constant depending on the degree $n$ of algebraic numbers.
Expression $f \asymp g$ is used for asymptotically equivalent functions, that is, $g \ll f \ll g$.
Notation $f \ll_{x_1,x_2,\ldots}\, g$ implies that the implicit constant depends only on parameters $x_1,x_2,\ldots$. Asymptotic equivalence $f \asymp_{x_1,x_2,\ldots}\, g$ is defined by analogy.

In the paper, we assume that the degree $n$ is arbitrary but fixed, and the parameter~$Q$, which bounds heights of polynomials and numbers, tends to infinity.

Note that we consider all algebraic numbers as complex numbers, i.e. elements of $\mathbb{C}$.

\subsection{Background}

The paper was motivated by results from two overlapping research areas of Diophantine approximation.
Both these areas widely employ the height \eqref{eq-n-height-def} as a height function.

One of these areas deals with
sets of numbers well approximable by algebraic numbers
and involves well-spaced subsets of algebraic numbers known as regular systems.
The idea of a regular system was developed as a useful tool for calculating the Hausdorff dimension in the paper by Baker and Schmidt~\cite{BakSch70}, who proved that the real algebraic numbers form a regular system. See \cite{BernDod1999} by Bernik and Dodson, and \cite{Bere99} by Beresnevich for gradual improving the spacing parameters of such regular systems. In~2002, Bugeaud~\cite{Bugeaud2002} proved that the real algebraic integers form a regular system too (with spacing parameters similar to the ones from \cite{Bere99}). In simple language, that is, there exists a constant $c_n$ depending on $n$ only such that for any interval $I\subseteq[-1,1]$ for all sufficiently large $Q\ge Q_0(I)$ there exist at least
$c_n |I| Q^n$
algebraic integers $\alpha_1,\dots,\alpha_k \in I$ of degree $n$ and height at most $Q$ such that the distances between them are at least
$Q^{-n}$.
See \cite{Bug2004} and \cite{BernDod1999} for some history and further references on the use of regular systems in calculation of the Hausdorff dimension of some sets.

Another area is concerned with the theory of Farey sequences and their generalizations.
In~1971, Brown and Mahler \cite{BroMah1971} suggested a generalization of the Farey sequences for algebraic numbers of higher degrees
and posed several questions about these sequences.
According to \cite{BroMah1971}, the \emph{$n$-th degree Farey sequence of order $Q$} is the sequence of all real roots of the set of integer polynomials of degree (at most) $n$ and height at most $Q$.
The elements of the 1-st degree sequence lying within $[0,1]$ form the well-known classical Farey sequence (see \cite{Koleda2017} for details)
and tend to be distributed uniformly in $[0,1]$ as $Q\to\infty$.
For $n\ge 2$ it turned out \cite{Koleda2017} that, as $Q$ gets large, the distribution of the $n$-th degree Farey sequence never tends to be uniform, however, it can be described in terms of a density function.
Namely, the following theorem was proved.
\begin{theorem}[\cite{Koleda2017}]\label{thm-algnum}
The number $\Phi_n(Q,I)$ of algebraic numbers of degree $n$ and height at most~$Q$ lying in an interval $I$
equals to
\[
\Phi_n(Q,I)=\frac{Q^{n+1}}{2 \zeta(n+1)} \int_I \phi_n(t) \,dt + O\left(Q^n (\ln Q)^{\ell(n)}\right),
\]
where $\zeta(\cdot)$ is the Riemann zeta function;
the implicit big-O-constant depends on the degree~$n$ only;
the power of the logarithm is equal to:
\[
\ell(n) =
\begin{cases}
1, & n\le 2,\\
0, & n\ge 3.
\end{cases}
\]
The function $\phi_n$ is given by the formula:
\begin{equation}\label{eq-phi-def}
  \phi_n(t) = \int\limits_{G_n(t)} \left|\sum_{k=1}^n k p_k t^{k-1}\right|\,dp_1\ldots\,dp_n, \qquad t \in \mathbb{R},
\end{equation}
where
\begin{equation}\label{eq-G-def}
G_n(t) = \left\{(p_1,\ldots,p_n) \in \mathbb{R}^n : \ \max\limits_{1\le k\le n}|p_k| \le 1, \ \left| \sum_{k=1}^n p_k t^k \right| \le 1 \right\},
\end{equation}
and satisfies the following function equations:
\begin{equation}\label{eq-phi-eq}
\phi_n(-t) = \phi_n(t), \qquad \phi_n(t^{-1}) = t^2 \phi_n(t).
\end{equation}
\end{theorem}

The expression \eqref{eq-phi-def} defines a continuous positive piecewise function. For example \cite{Koleda2015}, $\phi_2(t)$ is a piecewise rational function.
If $|t|\le 1-\frac{1}{\sqrt{2}}\approx 0{,}29$, it is possible \cite[Remark 2 in Section 4]{Koleda2017} to represent \eqref{eq-phi-def} by an explicit analytic expression for all $n$:
\begin{equation*}\label{eq-analit-f}
\phi_{n}(t) = 2^{n-1} \left(1+\frac13\sum_{k=1}^{n-1} (k+1)^2 t^{2k}\right).
\end{equation*}
From \eqref{eq-phi-eq} (using the second functional equation) one can easily obtain for $|t|\ge 2+\sqrt{2}\approx 3{,}41$ that
\begin{equation}\label{eq-analit-f2}
\phi_{n}(t) = \frac{2^{n-1}}{t^2} \left(1+\frac13\sum_{k=1}^{n-1} \frac{(k+1)^2} {t^{2k}}\right).
\end{equation}

It is worth to notice that the function $\phi_n$ from Theorem \ref{thm-algnum} coincides (up to a constant factor) with the density function of real zeros of $G_n$ (cf. \cite{Za05}).
Another formula representing $\phi_n$ can be deduced from \cite{Kac1949}.

\subsection{Main results}

The arranging of the real algebraic numbers into the generalized Farey sequence \cite{BroMah1971} suggests a way of ordering and counting real algebraic integers.

Let $\mathcal{O}_n$ denote the set of algebraic integers of degree $n$ (over $\mathbb{Q}$).
For a set $S \subseteq \mathbb{R}$,
let $\Omega_n(Q,S)$ be the number of algebraic integers $\alpha \in S$ of degree $n$ and height at most $Q$:
\[
\Omega_n(Q,S) := \#\left\{\alpha \in \mathcal{O}_n \cap S : H(\alpha)\le Q \right\}.
\]

Note that the algebraic integers of degree 1 are simply the rational integers, which are nowhere dense in the real line. Therefore, we assume $n\ge 2$.

We prove the following two theorems.

\begin{theorem}\label{thm-main}
Let $n\ge 2$ be a fixed integer. And let $I\subseteq\mathbb{R}$ be an interval.
Then, as $Q\to\infty$,
\begin{equation}\label{eq-main}
\frac{\Omega_n(Q,I)}{(2Q)^n} = 2^{-n} \int_I \widetilde\omega_n(Q^{-1},t)\,dt + r_n(Q,I),
\end{equation}
where the function $\widetilde\omega_n(\xi,t)$ has the form
\begin{equation}\label{eq-omega}
\widetilde\omega_n(\xi, t) = \begin{cases}
\phi_{n-1}(t) + 2^{n-2}\xi^2 t^2, & |t|\le \xi^{-1/2},\\
2^{n-1}\xi, & \xi^{-1/2}<|t| < \xi^{-1}+1,\\
0, & |t|\ge\xi^{-1}+1,
\end{cases}
\end{equation}
with the same $\phi_n(t)$ as in \eqref{eq-phi-def};
the remainder term $r_n(Q,I)$ satisfies the estimate
\begin{equation}\label{eq-rem-main}
r_n(Q,I) = \begin{cases}
O(Q^{-1}), & n\ge 3,\\
O\!\left(\frac{\ln Q}{Q}\right), & n=2,\\
\end{cases}
\end{equation}
where the implicit big-O-constants depend only on~$n$.
\end{theorem}

Note that in \eqref{eq-omega} the bounding values $\xi^{-1/2}$ and $\xi^{-1}+1$
are so tidy due to the possibility to hide roughnesses in the remainder term $r_n(Q,I)$.
The interested reader can track down the initial bounding values from Lemma~\ref{lm-delta} (Section~\ref{sbsbsec-proof-idiff}).

\medskip

{\sc Remark 1.}
The estimate \eqref{eq-rem-main} generally cannot be improved much.
Lemma~\ref{lm-empty-int} (Section~\ref{sec-aux} below) provides examples of intervals $I$ for which $r_n(Q,I)\asymp Q^{-1}$.
Regarding the asymptotic order of $r_n(Q,I)$, only two minor improvements are achievable in general settings.
Firstly, Lemma~\ref{lm-empty-int} shows that ${\Omega_n(Q,S) = 0}$ for any set $S\subset\mathbb{R}\setminus({-Q-1},{Q+1})$, and therefore, $r_n(Q,S) = 0$.
Secondly, for $n=2$ it can be proved \cite{Koleda2016} that in \eqref{eq-rem-main} we actually have
\begin{equation*}
r_2(Q,I)= - 2Q^{-1} \int\limits_{I\cap[-Q,Q]} \frac{dt}{\max(1,|t|)} + O\left(Q^{-1}\right).
\end{equation*}

\medskip

{\sc Remark 2.}
If $I\subset(-Q-1,-Q^{1/2})\cup(Q^{1/2},Q+1)$, then the equalities~\eqref{eq-main}, \eqref{eq-omega} and~\eqref{eq-rem-main} give
\begin{equation}\label{eq-main2}
\Omega_n(Q,I) = 2^{n-1} Q^{n-1} \left(|I| + O\!\left(\ln^{\ell(n)} Q\right)\right),
\end{equation}
where
the implicit big-O-constant depends only on~$n$.
The equation \eqref{eq-main2} shows that some relatively large subset of real algebraic integers statistically behaves just like the rational integers, which satisfy for any $I\subseteq(-Q-1,Q+1)$
\[
\Omega_1(Q,I) = |I| + O(1).
\]
The asymptotics \eqref{eq-main2} becomes nontrivial if the interval $I$ is large enough, namely, when
\[
\lim_{Q\to\infty} |I| (\ln Q)^{-\ell(n)} = \infty.
\]
Besides, these two ``uniform parts'' are unstable as $Q$ grows: they move away from the coordinate origin and spread wider, so that a fixed point can belong to a ``uniform part'' only for a finite range of $Q$.

One can show (see Lemma \ref{lm-Perron-n} below) that any algebraic integer $\alpha$ of degree $n$ and height at~most $Q$ satisfying $\alpha>(n+1)^{1/4} Q^{1/2}$ is a Perron number.
So, these two symmetric ``plateaux'' in the distribution of real algebraic integers are formed mainly from Perron numbers and their negatives.

\medskip

\begin{theorem}\label{thm-total}
Let $n\ge 2$.
Then, as $Q\to\infty$,
\begin{equation}\label{eq-total}
\frac{\Omega_n(Q,\mathbb{R})}{(2Q)^n} = 2^{-n}\int_{\mathbb{R}} \phi_{n-1}(t) dt +1-\frac4{3\sqrt{Q}} + r_n(Q,\mathbb{R}),
\end{equation}
where
\[
r_n(Q,\mathbb{R}) = \begin{cases}
O(Q^{-1}), & n\ge 3,\\
- \frac{\ln Q}{Q}  + \frac{2(1-\gamma)}{Q} + O\!\left(\frac1{Q\sqrt{Q}}\right), & n=2.
\end{cases}
\]
Here $\gamma = 0{,}5772\dots$ is Euler's constant; the big-O-constants depend only on $n$ (for $n=2$ the implicit constant is absolute).
\end{theorem}

It is worth to emphasize an interesting (and a bit surprising) feature of the distribution of real integers.
From \eqref{eq-omega} and \eqref{eq-analit-f2} one can easily see that for all $t$
\begin{equation}\label{eq-un-diff}
\left|\widetilde\omega_n(\xi, t) - \phi_{n-1}(t)\right| \le 2^{n-1} \xi,
\end{equation}
that is, the function $\omega_n(\xi,t)$ uniformly converges to $\phi_{n-1}(t)$ as $\xi$ tends to zero.
Therefore, if we take any finite fixed interval $I$, we get from \eqref{eq-un-diff}
\[
\lim_{\xi\to 0} \frac{\int_I \widetilde\omega_n(\xi,t) dt}{\int_I \phi_{n-1}(t) dt} = 1.
\]
Hence Theorem \ref{thm-main} shows that the real algebraic integers of degree $n$ statistically behave much like real algebraic numbers of degree $n-1$ as their heights grow (in this regard, it is interesting to compare the main results of \cite{Bere99} and \cite{Bugeaud2002}).
So, in view of the aforesaid, Theorem \ref{thm-main} could say that
the limit density function of real algebraic integers of degree $n$ is equal to the density function of real algebraic numbers of degree~${(n-1)}$.

But the surprise appears if we assume $I=\mathbb{R}$. Then from \eqref{eq-main} and \eqref{eq-total}
\[
\lim_{\xi\to 0} \frac{\int_{\mathbb{R}} \widetilde\omega_n(\xi,t) dt}{\int_{\mathbb{R}} \phi_{n-1}(t) dt} =
1+2^n\left(\int_{\mathbb{R}} \phi_{n-1}(t) dt\right)^{-1} > 1.
\]

Thus Theorems \ref{thm-main} and \ref{thm-total} yield the following corollary.

\begin{corollary}
For any fixed finite interval $I\subset\mathbb{R}$ the following limit equality is true
\begin{equation}\label{eq-ratio-f}
\lim_{Q\to\infty} \frac{\Omega_n(Q,I)}{\Phi_{n-1}(Q,I)} = 2\zeta(n).
\end{equation}
However, if $I$ is infinite or may depend on $Q$, the equality \eqref{eq-ratio-f} generally does not hold.
In particular,
\[
\lim_{Q\to\infty} \frac{\Omega_n(Q,\mathbb{R})}{\Phi_{n-1}(Q,\mathbb{R})}
= 2\zeta(n)\left(1+\frac{2^n}{\int_{\mathbb{R}} \phi_{n-1}(t)\,dt}\right).
\]
\end{corollary}

\subsection{Counting results for other heights}

For fair exposition, we should mention a number of counting results with respect to height functions other than the na\"{\i}ve height \eqref{eq-n-height-def}.
These results can be described by the following scheme.

Let $\mathbb{S}$ be the set of all algebraic elements of some sort (e.g. the algebraic numbers of a fixed degree, a number field, a ring of algebraic integers, a group of algebraic units, etc.). One defines a height function $H:\mathbb{S}\to \mathbb{R}$ such that
the value
\begin{equation}\label{eq-N-def}
N(\mathbb{S},X) := \#\left\{\alpha\in\mathbb{S}: H(\alpha)\le X \right\}
\end{equation}
is finite for all $X < +\infty$.

Then one may ask the question about the asymptotics of $N(\mathbb{S},X)$ as $X$ tends to infinity.
In most settings, known answers on the question look like
\begin{equation}\label{eq-gen-asymp}
N(\mathbb{S},X) = c\, X^k + O(X^{k-\gamma}),
\end{equation}
where the real parameters $c$, $k$, $\gamma$ and the big-O-constant depend only on the set $\mathbb{S}$.

As the function $H$, most papers on the subject employ the absolute Weil height, its generalizations or related functions.
In the simpliest setting, namely when one counts algebraic elements over $\mathbb{Q}$, the absolute Weil height $\mathcal{H}$ can be defined in terms of the Mahler measure as $\mathcal{H}(\alpha) = M(\alpha)^{1/n}$.

If $\alpha_1,\dots,\alpha_n$ are the roots of $p$, the \emph{Mahler measure} $M(p)$ of the polynomial can be defined as
\[
M(p) = |a_n| \prod_{i=1}^n \max(1,|\alpha_i|).
\]
For an algebraic number $\alpha$, its \emph{Mahler measure} $M(\alpha)$ is defined as the Mahler measure of the corresponding minimal polynomial.

Values $c$ in \eqref{eq-gen-asymp} have good-looking explicit expressions for several situations. See \cite{GriGu2017} for a nice account on the subject.
Some references and results can be found in the book by Lang \cite[chapter 3, \S 5]{Lang1983}.

For $\mathbb{S}$ being the set of algebraic numbers of degree $n$ over a fixed number field, and $H(\alpha)$ defined as the absolute Weil height of $\alpha$, such asymptotic formulas (with explicit constants $c$) are obtained by Masser and Vaaler \cite{MasVaa2008}, \cite{MasVaa2007}.
In 2001, Chern and Vaaler \cite[Theorem 6]{ChernVaaler2001} proved asymptotic estimates for the number of integer monic polynomials of degree $n$ having the Mahler measure bounded by $T$, which tends to infinity.
For the set $\mathcal{O}_n$ of algebraic integers of degree $n$ over $\mathbb{Q}$, and for $H$ being the Weil height,
from \cite{ChernVaaler2001} we have immediately
\[
N(\mathcal{O}_n, X) = c_n \,X^{n^2} + O\!\left(X^{n^2 - 1}\right),
\]
where $c_n$ is an explicit positive constant; in the big-O-notation the implicit constant depends only on $n$. Note that here $X$ has order of $Q^{1/n}$, where $Q$ is the upper bound for corresponding na\"{\i}ve heights.
In 2013, Barroero \cite{Barroero2013} extended this result to arbitrary ground number fields and improved the remainder term to $O\!\left(X^{n^2 - n}\right)$ for $\mathcal{O}_n$.

In 2016, Grizzard and Gunther \cite{GriGu2017} obtained an asymptotics like \eqref{eq-gen-asymp} for $\mathbb{S}$ being the set of all such algebraic numbers of degree $n$ over $\mathbb{Q}$ that their minimal polynomials all have the same specified leftmost and rightmost coefficients.
This approach gives a unified way to count algebraic numbers, integers and units over $\mathbb{Q}$.
In \cite{GriGu2017} one can also find explicit bounds on the error terms in the aforementioned results by Chern and Vaaler, Masser and Vaaler, and Barroero.

Widmer \cite{Wi2016} obtained a multiterm asymptotics of $N(\mathcal{O}_{\mathbb{K}}(k,m),X)$ for the set $\mathcal{O}_{\mathbb{K}}(k,m)$ of such $k$-tuples of algebraic integers that the coordinates of every point together generate a number field of a fixed degree $m$ over a given finite extension $\mathbb{K}$ of $\mathbb{Q}$.

In \cite{CaHu2017}, Calegari and Huang calculated the asymptotic number \eqref{eq-N-def} of algebraic integers~$\alpha$ (including Perron numbers, totally real and totally complex algebraic integers) of a fixed degree with the height function $H$ defined as the maximum absolute value of the roots of the minimal polynomial of $\alpha$. For a Perron number $\alpha$, this height function is merely the absolute value $|\alpha|$.

\subsection{Outline of the paper}

Now we give a short outline of the paper.
Section~\ref{sec-aux} contains neccessary auxiliary statements, and in the first reading one can skip it.
In Section~\ref{sec-mainproof} we prove a counterpart of Theorem~\ref{thm-algnum} for algebraic integers forming a ground for deriving Theorem~\ref{thm-main}.
Section~\ref{sec-limit} is devoted to the proof of Theorems~\ref{thm-main} and \ref{thm-total}.
In~Subsection~\ref{sec-2deg}, for the reader's convenience, we recall relevant facts about quadratic algebraic integers from \cite{Koleda2016}.
In~Subsection~\ref{sec-ndeg}, we treat algebraic integers of degree $n\ge 3$.
This separation is caused by inapplicability of the proof for $n\ge 3$ to $n=2$.
However, in their final form, the general results cover all degrees $n\ge 2$.

\section{Preliminary lemmas}\label{sec-aux}

\begin{lemma}[\cite{Chela1963}]\label{lm-red-m-pol}
Let $\mathcal{R}_n(Q)$ denote the number of reducible integer monic polynomials of degree $n$ and height at most $Q$. Then
\[
\lim\limits_{Q\to \infty} \frac{\mathcal{R}_n(Q)}{Q^{n-1}} = \upsilon_n, \qquad
\lim\limits_{Q\to \infty} \frac{\mathcal{R}_2(Q)}{2Q \ln Q} = 1,
\]
where $\upsilon_n$ is an effective positive constant depending on $n$ only, $n \ge 3$.
\end{lemma}

\begin{lemma}[\cite{Dav51_LP}]\label{lm-int-p-num}
Let $\mathcal{D}\subset \mathbb{R}^d$ be a bounded region formed by points $(x_1,\dots,x_d)$ satisfying a finite collection of algebraic inequalities
\[
F_i(x_1,\dots,x_d)\ge 0, \qquad 1\le i\le k,
\]
where $F_i$ is a polynomial of degree $\deg F_i \le m$ with real coefficients.
Let
\[
\Lambda(\mathcal{D}) = \mathcal{D}\cap \mathbb{Z}^d.
\]
Then
\[
\left|\#\Lambda(\mathcal{D}) - \mes_d \mathcal{D}\right| \le C \max(\bar{V}, 1),
\]
where the constant $C$ depends only on $d$, $k$, $m$; the quantity $\bar{V}$ is the maximum of all $r$--dimensional measures of projections of $\mathcal{D}$ onto all the~coordinate subspaces obtained by making $d-r$ coordinates of points in $\mathcal{D}$ equal to zero, $r$ taking all values from $1$ to $d-1$, that is,
\[
\bar{V}(\mathcal{D}) := \max\limits_{1\le r < d}\left\{ \bar{V}_r(\mathcal{D}) \right\}, \quad
\bar{V}_r(\mathcal{D}) := \max\limits_{\substack{\mathcal{J}\subset\{1,\dots,d\} \\ \#\mathcal{J} = r}}\left\{ \mes_r \operatorname{Proj}_{\mathcal{J}} \mathcal{D} \right\},
\]
where $\operatorname{Proj}_{\mathcal{J}} \mathcal{D}$ is the orthogonal projection of $\mathcal{D}$ onto the coordinate subspace formed by coordinates with indices in $\mathcal{J}$.
\end{lemma}

Remark. There are several more general and more recent results on counting integer points in regions (see e.g. \cite{BaWi2014} and references there). However, for our purposes, this Davenport's lemma is enough.

\begin{lemma}\label{lm-jacob}
Let $n\ge 2$. Let $\xi$ be a fixed positive real number. Let vectors $(a_{n-1}, \dots, a_1, a_0)$ and $(b_{n-2},\dots, b_1, b_0, \alpha, \beta)$ be related by the equality
\begin{equation}\label{eq-p-roots}
\xi x^n + \sum\limits_{k=0}^{n-1} a_k x^k =
(x-\alpha)(x-\beta)\left(\xi x^{n-2}+ \sum\limits_{m=0}^{n-3} b_m x^m \right).
\end{equation}
Then this relation can be expressed in the following matrix form:
\begin{equation}\label{eq_zamena2_matr}
\left(\begin{array}{l}
\xi\\
a_{n-1}\\
a_{n-2}\\
\vdots\\
a_1\\
a_0
\end{array}\right) =
\left(\begin{array}{cccc}
1 & & \phantom{\ddots} & \bigzero\\
-(\alpha+\beta) & 1 & \phantom{\ddots} &\\
\alpha\beta & -(\alpha+\beta) & \ddots &\\
 & \alpha\beta & \phantom{\ddots} & 1\\
 & & \ddots & -(\alpha+\beta) \\
\bigzero & & \phantom{\ddots} & \alpha\beta
\end{array}\right) \cdot
\left(\begin{array}{l}
\xi \\
b_{n-3}\\
\vdots\\
b_1\\
b_0
\end{array}\right),
\end{equation}
and the Jacobian of this coordinate change is equal to
\begin{equation*}
\det J = \left| \frac{\partial (a_{n-1},\ldots,a_2,a_1,a_0)}{\partial (b_{n-3},\ldots,b_0;\alpha,\beta)} \right| =
(\beta-\alpha)\cdot g(\mathbf{b}, \alpha) \cdot g(\mathbf{b}, \beta),
\end{equation*}
where $g(\mathbf{b}, x) := \xi x^{n-2} + b_{n-3} x^{n-3} + \ldots + b_1 x + b_0$.
\end{lemma}
Note that for $n=2$ the Jacobian equals to $\xi^2(\beta-\alpha)$.

\begin{proof}
Note that in \eqref{eq_zamena2_matr} the left hand side vector has $n+1$ coordinates, the right hand side vector has $n-1$ coordinates, and the dimension of the matrix is $(n+1)\times(n-1)$.
We just discard the first row of the l.h.s. vector and of the matrix from our consideration because $\xi$ is a constant in the settings of the lemma.

It is not hard to check that both \eqref{eq-p-roots} and \eqref{eq_zamena2_matr} give the same expressions for $a_i$.
To~calculate the Jacobian, we need to know the derivatives $\frac{\partial a_i}{\partial b_j}$, $\frac{\partial a_i}{\partial \alpha}$ and $\frac{\partial a_i}{\partial \beta}$.
To find them, we will differentiate the equality \eqref{eq_zamena2_matr} w.r.t. the corresponding variables.

For simplicity denote $\mathbf{a}=(a_{n-1},\dots,a_1,a_0)^T \in\mathbb{R}^n$. Then the Jacobian matrix takes the form
\[
J = \left(\frac{\partial \mathbf{a}}{\partial b_{n-3}}, \dots, \frac{\partial \mathbf{a}}{\partial b_0}, \frac{\partial \mathbf{a}}{\partial \alpha}, \frac{\partial \mathbf{a}}{\partial \beta} \right).
\]

For $b_j$ (where $0\le j\le n-3$), we have
\begin{equation}\label{eq-dbj}
\frac{\partial \mathbf{a}}{\partial b_j} = \left(\underbrace{0,\dots,0}_{n-3-j}, 1, -(\alpha+\beta), \alpha\beta, \underbrace{0,\dots,0}_j \right)^T.
\end{equation}

For $\alpha$ and $\beta$ after differentiation we obtain
\begin{align*}
\frac{\partial \mathbf{a}}{\partial \alpha}
=& - (\xi, b_{n-3},\dots, b_1,b_0, 0)^T + \beta (0, \xi, b_{n-3},\dots, b_1,b_0)^T,\\
\frac{\partial \mathbf{a}}{\partial \beta}
=& - (\xi, b_{n-3},\dots, b_1,b_0, 0)^T + \alpha (0, \xi, b_{n-3},\dots, b_1,b_0)^T.
\end{align*}

The value of the Jacobian will not change, if we replace the columns $\frac{\partial \mathbf{a}}{\partial \alpha}$ and $\frac{\partial \mathbf{a}}{\partial \beta}$ by the vectors
\begin{align}
\frac{\alpha}{\beta-\alpha}
\frac{\partial \mathbf{a}}{\partial \alpha} -
\frac{\beta}{\beta-\alpha}
\frac{\partial \mathbf{a}}{\partial \beta} &= (\xi, b_{n-3},\dots, b_1,b_0, 0)^T,  \label{eq-dal-dbe-1}\\
\frac{\partial \mathbf{a}}{\partial \alpha} - \frac{\partial \mathbf{a}}{\partial \beta} &= (\beta-\alpha)(0, \xi, b_{n-3},\dots, b_1,b_0)^T. \label{eq-dal-dbe-2}
\end{align}
Let us briefly explain this move.
This replacement is equivalent to right multiplication the Jacobian matrix $J$ by the block diagonal matrix $B$
\[
B = \diag(E_{n-2}, M), \qquad
M=\left(\begin{matrix}
\frac{\alpha}{\beta-\alpha} & 1\\
-\frac{\beta}{\beta-\alpha} & -1
\end{matrix}\right),
\]
where $E_{n-2}$ is the identity matrix of size $n-2$. It is easy to see that $\det(B)=1$.

Arranging all the vectors \eqref{eq-dbj}, \eqref{eq-dal-dbe-1}, \eqref{eq-dal-dbe-2} in one $(n\times n)$-determinant we get the following formula for the Jacobian.
\begin{equation*}
\det J = \det JB =
(\beta-\alpha) \cdot \left|
\begin{array}{cccccc}
1 & & \phantom{\ddots} & & b_{n-2} & \\
-(\alpha+\beta) & 1 & \phantom{\ddots} & & b_{n-3} & b_{n-2}\\
\alpha\beta & -(\alpha+\beta) & \ddots & & \vdots & b_{n-3}\\
 & \alpha\beta & \ddots & 1 & b_1 & \vdots\\
 & & \ddots & -(\alpha+\beta) & b_0 & b_1\\
 & & \phantom{\ddots} & \alpha\beta & & b_0
\end{array}
\right|.
\end{equation*}
One can observe that the determinant in the right-hand side is equal to the resultant $R(f,g)$ of the polynomials $f(x)=(x-\alpha)(x-\beta)$ and $g(x)= \xi x^{n-2} + b_{n-3} x^{n-3} +\ldots+b_1 x + b_0$.
This proves the lemma since $R(f,g) = g(\alpha)g(\beta)$ (see, e.g. \cite[\S 5.9]{Waer1970-1-en}).
\end{proof}

\begin{lemma}\label{lm-2roots}
Let $I = [a,b)\subset \mathbb{R}$ be a finite interval, $|I|\le 1$, and let $0<\xi \le 1$.
Let $\mathcal{M}_n(\xi, I)$ be the set of polynomials $p\in\mathbb{R}[x]$ with height $H(p) \le 1$ and $\deg(p(x) - \xi x^n) < n$ that have at least $2$ roots in~$I$.
Then
\begin{equation*}
\mes_{n} \mathcal{M}_n(\xi, I) \le \lambda(n) \left(\xi + \rho^{-3}\right)^2 |I|^3,
\end{equation*}
where $\rho = \max(1, |a+b|/2)$, and $\lambda(n)$ is a constant depending only on $n$.
\end{lemma}

\begin{proof}
To simplify notation, we use $\mathcal{M} := \mathcal{M}_n(\xi, I)$.
Estimate from above the measure
\begin{equation*}
\mes_n \mathcal{M} = \int\limits_{\mathcal{M}} d{\bf a}.
\end{equation*}
Every polynomial $p(x)$ in $\mathcal{M}$ can be expressed in the form
\[
p(x) = \xi x^n + a_{n-1} x^{n-1} + \ldots + a_0 =
(x-\alpha)(x-\beta)(\xi x^{n-2}+ b_{n-3} x^{n-3}+\ldots + b_0),
\]
where $\alpha, \beta \in I$.

Change the coordinates by \eqref{eq_zamena2_matr}.
The condinition ${\bf a} \in \mathcal{M}$ is equivalent to the system of inequalities
\begin{equation}\label{eq_grM2}
\left\{
\begin{array}{l}
|a_{n-1}| = |b_{n-3} - (\alpha+\beta) \xi| \le 1,\\
|a_{n-2}| = |b_{n-4}-(\alpha+\beta)b_{n-3}+\alpha\beta \xi| \le 1,\\
|a_k| = |b_{k-2} - (\alpha+\beta) b_{k-1} + \alpha\beta b_k| \le 1, \ \ \ k = 2,\ldots,n-3,\\
|a_1| = |- (\alpha+\beta) b_0 + \alpha\beta b_1| \le 1,\\
|a_0| = |\alpha b_0| \le 1,\\
a\le \alpha < b, \\
a\le \beta < b.
\end{array}
\right.
\end{equation}

From Lemma \ref{lm-jacob}, we obtain
\begin{equation}\label{eq-M-ineq}
\mes_n \mathcal{M} \le \int\limits_{\mathcal{M}^*} |\alpha-\beta| \cdot |g(\mathbf{b}, \alpha)\; g(\mathbf{b},\beta)|\, d\mathbf{b}\, d\alpha\, d\beta,
\end{equation}
where $\mathcal{M}^*$ is the new integration domain defined by the inequlities \eqref{eq_grM2}, here $g(\mathbf{b},x) = \xi x^{n-2}+\ldots+b_1 x + b_0$. Note that here we have inequality instead of equality. The reason is that a polynomial having $k>2$ roots in $I$ can be written in $\binom{k}{2}$ different ways in the form~\eqref{eq-p-roots}.

Write the multiple integral \eqref{eq-M-ineq} in the following manner:
\[
\mes_n \mathcal{M} \le \int\limits_{I\times I} |\alpha-\beta|\, d\alpha\, d\beta \int\limits_{\mathcal{M}^*(\alpha,\beta)} |g(\mathbf{b},\alpha) g(\mathbf{b},\beta)|\,d\mathbf{b},
\]
where $\mathcal{M}^*(\alpha,\beta)$ is the set of vectors $\mathbf{b} \in \mathbb{R}^{n-1}$ that satisfy the inequalities~\eqref{eq_grM2}.

Estimate the internal integral using upper bounds on $\mes_{n-2} \mathcal{M}^*(\alpha,\beta)$ and on the function $G({\bf b},\alpha,\beta):= g(\mathbf{b},\alpha)\; g(\mathbf{b},\beta)$ for $\mathbf{b} \in \mathcal{M}^*(\alpha,\beta)$. Consider the two cases.

1) Let $|a+b|/2 \le 1$.

Estimate the measure $\mes_{n-2} \mathcal{M}^*(\alpha,\beta)$ using the submatrix equation from~\eqref{eq_zamena2_matr}:
\begin{equation}\label{eq-matr1}
\left(\begin{array}{l}
a_{n-1} - c_1\xi\\
a_{n-2} - c_0\xi\\
\vdots\\
a_3\\
a_2\\
\end{array}\right) =
\left(\begin{array}{lllll}
c_2 & & & & \bigzero\\
c_1 & c_2 & &\\
c_0 & c_1 & \ddots &\\
\vdots & \vdots & \ddots & c_2\\
 & & \dots & c_1 & c_2
\end{array}\right) \cdot
\left(\begin{array}{l}
b_{n-3}\\
b_{n-4}\\
\vdots\\
b_1\\
b_0
\end{array}\right),
\end{equation}
where $c_0 = \alpha\beta$, $c_1 = -(\alpha + \beta)$, $c_2 = 1$.

Multiplication by the $(n-2)\times(n-2)$-matrix \eqref{eq-matr1} maps the region $\mathcal{M}^*(\alpha, \beta)$ into a parallelepiped of the unit volume.
The determinant of the matrix \eqref{eq-matr1} is equal to 1.
Hence, we have the upper bound:
\[
\mes_{n-2} \mathcal{M}^*(\alpha,\beta) \le 1.
\]

Estimate $|G(\mathbf{b},\alpha,\beta)|$ from above for $\mathbf{b}\in \mathcal{M}^*(\alpha,\beta)$. For this sake, find upper bounds on the coordinates $b_{n-3}, \dots, b_1, b_0$ in the region $\mathcal{M}^*(\alpha, \beta)$.

Since $|a+b|/2 \le 1$ and $0 < b-a \le 1$, we have the following estimate for the matrix coefficients in \eqref{eq-matr1}:
\[
\max\limits_{0\le i\le 2}|c_i| = O(1).
\]
Starting from $b_{n-3}$, we express the coefficients $b_i$ and obtain by induction
\begin{equation*}
\max\limits_{0\le i\le n-3} |b_i| \ll_n 1.
\end{equation*}
Hence, we have that $|G({\bf b},\alpha,\beta)| \ll_n 1$ for all $\mathbf{b}\in \mathcal{M}^*(\alpha, \beta)$.
So we obtain for $\alpha,\beta\in [a,b)$
\[
\int\limits_{\mathcal{M}^*(\alpha,\beta)}|G({\bf b},\alpha,\beta)|\,d{\bf b} \ \ll_n \ 1.
\]
Therefore, for $|a+b|/2 \le 1$ we have
\begin{equation}
\mes_n \mathcal{M} \ll_n |I|^3.
\end{equation}

2) Let $|a+b|/2 > 1$.

From \eqref{eq_zamena2_matr} we have
\begin{equation}\label{eq-matr2}
\left(\begin{array}{l}
a_{n-3}\\
a_{n-4}\\
\vdots\\
a_1\\
a_0
\end{array}\right) =
\left(\begin{array}{lllll}
c_0 & c_1 & \dots & & \\
 & c_0 & \ddots & \vdots & \vdots \\
 & & \ddots & c_1 & c_2\\
 & & & c_0 & c_1\\
\bigzero & & & & c_0
\end{array}\right) \cdot
\left(\begin{array}{l}
b_{n-3}\\
b_{n-4}\\
\vdots\\
b_1\\
b_0
\end{array}\right),
\end{equation}
where $c_0 = \alpha\beta$, $c_1 = -(\alpha + \beta)$, $c_2 = 1$.

Hence, we obtain the upper bound
\[
\mes_{n-2} \mathcal{M}^*(\alpha,\beta) \le |\alpha\beta|^{-n+2}.
\]

Find an upper bound for $|G(\mathbf{b},\alpha,\beta)|$. For this sake, estimate $|b_i|$ from above.

The matrix coefficients $c_i$ in \eqref{eq-matr2} can be estimated in the following manner:
\[
c_0 \asymp \rho^2, \quad c_1 = O(\rho), \quad c_2 = 1,
\]
where $\rho = |a+b|/2$. Here, we take into account that $0< b-a \le 1$.

By induction, we estimate from above $|b_i|$. From \eqref{eq-matr2}, we have
\begin{align*}
|b_0| &=  |c_0|^{-1} |a_0| = O(\rho^{-2}), \\
|b_1| &= |c_0|^{-1} |a_1 - c_1 b_0| = O(\rho^{-2})\; O(1 + \rho^{-1}) = O(\rho^{-2}).
\end{align*}
Proceeding by induction for $i=2,3,\dots,n-3$, we obtain
\[
|b_i| = |c_0|^{-1} |a_i - c_1 b_{i-1} - c_2 b_{i-2}| = O(\rho^{-2})\; O(1 + \rho^{-1} + \rho^{-2}) = O(\rho^{-2}),
\]
where the implicit big-O-constants depend only on $n$.

Hence, for any $x\in [a,b)$ we obtain $|g(\mathbf{b},x)| = O(\xi \rho^{n-2} + \rho^{n-5})$, so
\[
|G(\mathbf{b},\alpha,\beta)| = O\left(\rho^{2(n-2)} \left(\xi+\rho^{-3}\right)^2\right).
\]
Therefore, we obtain
\[
\mes_n \mathcal{M} \le \lambda(n) \left(\xi+\rho^{-3}\right)^2 |I|^3.
\]
The lemma is proved.
\end{proof}

\begin{lemma}[see, e.g., \cite{Koleda2017}]\label{lm-empty-int}
Let $x_0=a/b$ with $a\in\mathbb{Z}$, $b\in\mathbb{N}$ and $\gcd(a,b)=1$.
Then there are no algebraic numbers $\alpha$ of degree $\deg \alpha = n$ and height $H(\alpha)\le Q$ in the interval $|x-x_0|\le r_0$, where
\[
r_0=r_0(x_0, Q)=\frac{\kappa(n)}{\max(|a|,b)^n Q},
\]
and $\kappa(n)$ is an effective constant depending only on $n$.

For a neighborhood of infinity: no algebraic number $\alpha$ of degree $\deg(\alpha)=n$ and height $H(\alpha)\le Q$ lies in the set $\{x\in\mathbb{R}: |x|\ge Q+1\}$.
\end{lemma}

Note that the statement of the lemma implies that
$\Omega_n(Q,S) = 0$ if $S \cap (-Q-1,Q+1) = \varnothing$, and
$\Omega_n(Q,S) = \Omega_n(Q,\mathbb{R})$ if $(-Q-1,Q+1)\subseteq S$.

\begin{proof}
For the reader convenience we give proof here.

Let $p(x)=a_n x^n+\ldots+a_1 x+a_0\in\mathbb{Z}[x]$ with $H(p)\le Q$, and let $p(x_0)\ne 0$.
Then
\begin{equation}\label{eq_pa}
|p(x_0)| \ge \frac{1}{b^n}.
\end{equation}

On the other hand, we have:
\begin{equation*}
|p(x)-p(x_0)| \le H(p) \sum_{k=1}^n \left|x^k - x_0^k\right| = |x-x_0| H(p) \sum_{k=1}^n \left|\sum_{j=0}^{k-1} x^j x_0^{k-1-j}\right|.
\end{equation*}

Assuming $|x_0|\le 1$, $|x|\le 1$ and $x\ne x_0$, we obtain
\begin{equation}\label{eq_dpa}
|p(x)-p(x_0)| < |x-x_0| H(p) \sum_{k=1}^n k = \frac{n(n+1)}{2} H(p) |x-x_0|.
\end{equation}
The estimates \eqref{eq_pa} and \eqref{eq_dpa} show that $x$ cannot be a root of $p(x)$ if $x$ is sufficiently close to $x_0$, namely, if $x$ satisfies the inequality
\[
\frac{n(n+1)}{2} H(p)\,|x-x_0| \le \frac{1}{b^n}.
\]
Denoting $\kappa(n) = \frac{2}{n(n+1)}$ we obtain the lemma for $|x_0|\le 1$, namely, no algebraic number~$\alpha$ of degree~$n$ and height~$H$ can satisfy the inequality
\begin{equation*}
|\alpha-x_0| \le \frac{\kappa(n)}{b^n H}.
\end{equation*}

Now we deal with the case $|x_0|>1$. Let us show how it can be reduced to the case $|x_0| < 1$.
Note that the set of all real/non-real/complex algebraic numbers of any fixed degree $n$ is invariant under the mapping $x\to x^{-1}$.
This is because if an integer polynomial $p(x)$ of degree $n$ is irreducible over $\mathbb{Q}$, then the polynomial $x^n p(x^{-1})$ also has
integer coefficients,
the same degree and height,
and is irreducible over $\mathbb{Q}$; and vice versa.

Suppose that $|x|>1$, and $x$ has the same sign as $x_0$. Then the interval between $\widetilde{x}_0=x_0^{-1}=b/a$ and $\widetilde{x}=x^{-1}$ contains no algebraic numbers if and only if the interval with the end points $x_0$ and $x$ is free of algebraic numbers too.
From the first part of the proof, we know this to happen when
\[
|\widetilde{x}-\widetilde{x}_0| \le \kappa(n) H^{-1} |a|^{-n},
\]
or equivalently, in terms of $x_0$ and $x$,
\[
|x-x_0| \le \kappa(n) H^{-1} |a|^{-n} x x_0.
\]
Since $x x_0>1$, there are surely no algebraic numbers $\alpha$ of degree $n$ and height $H$ such that
\[
|\alpha-x_0| \le \kappa(n) H^{-1} |a|^{-n}.
\]

Let $|\beta|\ge Q+1$. Then
\[
|p(\beta)| \ge |\beta|^n - Q|\beta|^{n-1} - \ldots - Q|\beta|-Q \ge 1.
\]
Thus, the number $\beta$ cannot be a root of the polynomial $p(x)$.
Besides, hence, one can obtain that there no algebraic numbers $\beta$ of height $H(\beta)\le Q$ such that $|\beta|\le (Q+1)^{-1}$.
The lemma is proved.
\end{proof}

\begin{lemma}\label{lm-Perron-n}
Let $\alpha$ be an algebraic integer of degree $n$ and height $H$.
If $\alpha > (n+1)^{1/4} H^{1/2}$, then $\alpha$ is a Perron number.
\end{lemma}

By the way, we prove a stronger statement: every real algebraic integer $\alpha$ satisfying $\alpha > (M(\alpha))^{1/2}$ is a~Perron number.

\begin{proof}
Let $\alpha_1,\dots,\alpha_n$ be the roots of a real monic polynomial $p$ of degree $n$ numbered so that
\[
|\alpha_1|\ge |\alpha_2|\ge \dots \ge |\alpha_n|.
\]
Since the Mahler measure of $p$ is equal to
$M(p) = \prod_{k=1}^n \max(1,|\alpha_i|)$,
we get
\[
M(p)\ge |\alpha_1| |\alpha_2|.
\]
On the other hand, we have (see Theorem 4.2.1 in \cite[p.~142]{Pr04})
\[
M(p) \le (n+1)^{1/2} H(p).
\]
Thus, for any algebraic integer $\alpha_1 > (n+1)^{1/4} H(\alpha_1)^{1/2}$, its conjugates $\alpha_2,\dots,\alpha_n$ satisfy
\[
\max_{2\le k\le n}|\alpha_k| \le \frac{(n+1)^{1/2} H(\alpha_1)}{\alpha_1} < \alpha_1,
\]
hence $\alpha_1$ is a Perron number.
\end{proof}

\section{Approaching to Theorem \ref{thm-main}}\label{sec-mainproof}

The aim of this section is to prove the following theorem, which can be regarded as an ``integer counterpart'' of Theorem \ref{thm-algnum}.

\begin{theorem}\label{thm-int-counterpart}
Let $n\ge 2$ be a fixed integer. For any interval $I\subseteq\mathbb{R}$ we have:
\begin{equation}\label{eq-O/omega}
\Omega_n(Q,I) = Q^n \int_I \omega_n(Q^{-1}, t) \,dt + O\!\left(Q^{n-1}(\ln Q)^{\ell(n)}\right),
\end{equation}
where the function $\omega_n(\xi, t)$ can be written in the form:
\begin{equation}\label{eq-omega-def}
\omega_n(\xi, t) = \int\limits_{D_n(\xi, t)} \left|n\xi t^{n-1} + \sum_{k=1}^{n-1} kp_k t^{k-1} \right| \,dp_1 \dots dp_{n-1},
\end{equation}
with
\[
D_n(\xi, t) = \left\{(p_1,\dots,p_{n-1})\in\mathbb{R}^{n-1} : \max_{1\le k\le n-1}|p_k|\le 1, \ \left|\xi t^n + \sum_{k=1}^{n-1} p_k t^k \right| \le 1\right\}.
\]

In the remainder term, the implicit constant depends only on the degree $n$.
Besides, there exist intervals, for which the error of this formula has the order $O(Q^{n-1})$.
\end{theorem}

Let $I = [\alpha, \beta)$ be a finite interval. Denote by $\mathcal{N}_n(Q,k,I)$, $\widetilde{\mathcal{N}}_n(Q,k,I)$ and $\mathcal{R}_n(Q,k,I)$ the number of (respectively) all, irreducible, and reducible integer monic polynomials of degree $n$ and height at most $Q$ having exactly $k$ roots in~$I$. It is easy to see that
\begin{equation}\label{eq-Omega-N}
\Omega_n(Q,I) = \sum_{k=1}^n k\, \widetilde{\mathcal{N}}_n(Q,k,I).
\end{equation}

Evidently, $\widetilde{\mathcal{N}}_n(Q,k,I) = \mathcal{N}_n(Q,k,I) - \mathcal{R}_n(Q,k,I)$.

Let $\mathcal{G}_n(\xi,k,S)$ be the set of real polynomials of degree $n$ and height at most~1 with the leading coefficient $\xi$ having exactly $k$ roots (with respect to multiplicity) in a set $S$. Then from Lemma \ref{lm-int-p-num}, we have
\begin{equation*}
\mathcal{N}_n(Q,k,I) = Q^n \mes_n \mathcal{G}_n(Q^{-1},k,I) + O\!\left(Q^{n-1}\right).
\end{equation*}
Lemma \ref{lm-red-m-pol} gives the following bound on $\mathcal{R}_n(Q,k,I)$:
\[
\mathcal{R}_n(Q,k,I) \le \mathcal{R}_n(Q) \ll_n Q^{n-1} (\ln Q)^{\ell(n)}.
\]
Hence, we have
\begin{equation}\label{eq-N-G}
\widetilde{\mathcal{N}}_n(Q,k,I) = Q^n \mes_n \mathcal{G}_n(Q^{-1},k,I) + O\!\left(Q^{n-1} (\ln Q)^{\ell(n)}\right),
\end{equation}
where in the big-O-notation the implicit constant depends only on $n$.

It is easy to see that for all $0<\xi\le 1$ the function
\begin{equation}\label{eq-Omega-G}
\widehat{\Omega}_n(\xi, S):= \sum_{k=1}^n k \mes_n \mathcal{G}_n(\xi,k,S)
\end{equation}
is additive and bounded on the set of all subsets $S\subseteq \mathbb{R}$.

Let us show that $\widehat{\Omega}_n(\xi, I)$ can be written as the integral of a continuous function over~$I$.
Let
\begin{equation}\label{eq-B-def}
\mathcal{B}(\xi, I) = \left\{{\bf p}\in\mathbb{R}^{n+1} : \deg(p(x)-\xi x^n) < n, \  p(\alpha)p(\beta)< 0, \ H(p)\le 1 \right\},
\end{equation}
where ${\bf p} = (\xi, p_{n-1},\ldots,p_1,p_0)$ is the vector of the coefficients of the polynomial $p(x)= \xi x^n+\ldots+p_1 x+p_0$, and $\xi = Q^{-1}$.
Obviously, every polynomial from $\mathcal{B}(\xi,I)$ has the odd number of roots in the interval~$I$.

In Lemma \ref{lm-2roots}, we have $\mathcal{M}_n(\xi, I) = \bigcup_{k=2}^n\mathcal{G}_n(\xi, k, I)$. Hence, it follows that
\begin{equation}\label{eq-Omega-D}
\widehat{\Omega}_n(\xi, I) = \mes_n \mathcal{B}(\xi, I) + O(|I|^3),
\end{equation}
where the implicit big-O-constant depends only on $n$.

Now we calculate
\[
\mes_n \mathcal{B}(\xi, I) = \int\limits_{\mathcal{B}(\xi, I)} dp_0\,dp_1 \dots dp_{n-1}.
\]
To do this, we represent $\mathcal{B}(\xi, I)$ in a more convenient form.
In order to simplify notation, let us introduce the symbol
\[
g_t = g_t(p_1,\dots,p_{n-1}) = \xi t^n + \sum_{k=1}^{n-1} p_k t^k.
\]
Then we have $p(t) = g_t(p_1,\dots,p_{n-1}) + p_0$.

The inequality $p(\alpha)p(\beta)< 0$ from \eqref{eq-B-def} is equivalent to one of the two inequality systems:
\[
\begin{cases}
g_\alpha + p_0 < 0,\\
g_\beta + p_0 > 0,
\end{cases}
\qquad \text{or} \qquad
\begin{cases}
g_\alpha + p_0 > 0,\\
g_\beta + p_0 < 0,
\end{cases}
\]
which can be joined into the single double inequality
\[
-\max(g_\alpha, g_\beta) \le p_0 \le -\min(g_\alpha, g_\beta).
\]
So the region $\mathcal{B}(\xi, I)$ can be defined by the following inequalities
\begin{equation}\label{eq-sys}
\begin{cases}
\max\limits_{0\le k\le n-1} |p_k| \le 1,\\
-\max(g_\alpha, g_\beta) \le p_0 \le -\min(g_\alpha, g_\beta).
\end{cases}
\end{equation}
In this system, $p_0$ must satisfy two interval constraints.

In the next step we estimate the set $\widetilde{D}$ of all $(p_1,\dots,p_{n-1})\in\mathbb{R}^{n-1}$ for which the system~\eqref{eq-sys} has solutions in $p_0$.
In fact, $\widetilde{D}$ is the orthogonal projection of $\mathcal{B}(\xi, I)$ on the hyperplane $p_0=0$.
Consider the regions
\[
D_* := D_n(\xi, \alpha)\cap D_n(\xi, \beta), \qquad
D^* := D_n(\xi, \alpha)\cup D_n(\xi, \beta),
\]
where
\[
D_n(\xi, t) := \left\{(p_1,\dots,p_{n-1})\in\mathbb{R}^{n-1} : \max_{1\le i\le n-1} |p_i|\le 1, \ \left|g_t(p_1,\dots,p_{n-1}) \right| \le 1\right\}.
\]
The inequalities $|g_\alpha| \le 1$ and $|g_\beta| \le 1$ hold for all $(p_1,\dots,p_{n-1})\in D_*$. For any $(p_1,\dots,p_{n-1})\not\in D^*$, the inequalities $|g_\alpha| > 1$ and $|g_\beta| > 1$ hold simultaneously, and so for sufficiently close $\alpha$ and $\beta$ the system of inequalities \eqref{eq-sys} is contradictory. Otherwise, there would be two possibilities: $g_\alpha g_\beta > 0$ or $g_\alpha g_\beta < 0$.
If $g_\alpha$ and $g_\beta$ had the same sign, then $|p_0|>1$; but from the same system we have $|p_0|\le 1$.
If $g_\alpha$ and $g_\beta$ had opposite signs, then the inequality $|g_\beta-g_\alpha| > 2$ would hold; but the difference $|g_\beta-g_\alpha|$ uniformly tends to zero for all $(p_1,\dots,p_{n-1})\in[-1,1]^{n-1}$ as $\alpha$ and $\beta$ converge.
So we proved that
\[
D_* \subseteq \widetilde{D} \subseteq D^*.
\]
Moreover, for any fixed $(p_1,\dots,p_{n-1})\in\widetilde{D}$ the measure of values $p_0$ defined by \eqref{eq-sys} does not exceed
\begin{equation*}
h_\xi(p_{n-1},\ldots,p_1):= |g_\beta - g_\alpha|.
\end{equation*}
And for every fixed $(p_1,\dots,p_{n-1})\in D_*$ this measure is equal to $h_\xi(p_{n-1},\ldots,p_1)$.
Therefore,
\begin{equation*}
\int\limits_{D_*} h_\xi(p_{n-1},\ldots,p_1)\,dp_{n-1}\ldots dp_1
\le \ \mes_n\mathcal{B}(\xi, I) \ \le
\int\limits_{D^*} h_\xi(p_{n-1}\ldots,p_1)\,dp_{n-1}\ldots dp_1.
\end{equation*}
Now the obvious inclusion $D_* \subseteq D_n(\xi,\alpha) \subseteq D^*$ implies
\[
\left|\mes_n\mathcal{B}(\xi, I)\, - \!\int\limits_{D_n(\xi, \alpha)} h_\xi(p_{n-1},\ldots,p_1)\,dp_{n-1}\ldots dp_1 \right| \le
\int\limits_{D^*\setminus D_*} h_\xi(p_{n-1}\ldots,p_1)\,dp_{n-1}\ldots dp_1.
\]
It is easy to show that the difference of $D^*$ and $D_*$ has a small measure for sufficiently close $\alpha$ and $\beta$:
\[
\mes_{n-1} (D^* \setminus D_*) = O(\beta-\alpha).
\]
As noted above, $h_\xi(p_{n-1},\ldots,p_1)\to 0$ as $\beta\to\alpha$. Thus, we obtain for all $\alpha\in\mathbb{R}$
\begin{equation*}
\mes_n \mathcal{B}(\xi, I)=\omega_n(\xi,\alpha)|I|+o(|I|),
\end{equation*}
where $\omega_n(\xi,t)$ is defined in \eqref{eq-omega-def}.

Hence, as $|I|\to 0$, from \eqref{eq-Omega-D} we obtain that
\[
\widehat{\Omega}_n(\xi,I) = \omega_n(\xi,\alpha)|I|+o(|I|),
\]
So we have
\[
\widehat{\Omega}_n(\xi,I)  = \int_I\omega_n(\xi,t)\,dt.
\]
Therefore, from \eqref{eq-Omega-N}, \eqref{eq-N-G} and \eqref{eq-Omega-G}, we obtain the main theorem. Lemma \ref{lm-empty-int} shows that there exist infinitely many intervals $I$, for which the error of the asymptotic formula \eqref{eq-O/omega} is of the order $O(Q^{n-1})$.
Theorem \ref{thm-int-counterpart} is proved.

\section{Proving Theorems \ref{thm-main} and \ref{thm-total}}\label{sec-limit}

\subsection{Quadratic algebraic integers}\label{sec-2deg}

In this subsection, to give a full picture, we recall some relevant facts about the density function $\omega_2(\xi,t)$ of real quadratic algebraic integers.

Quadratic integers are considered separately from the general case $n\ge 3$, since the techniques used for $n\ge 3$ do not work for $n=2$, because of the lower dimension of spaces arising in the proof.
\begin{theorem}[\cite{Koleda2016}]\label{thm-omega2}
For $\xi\le 1/4$,
\begin{equation*}
\omega_2(\xi, t) = \begin{cases}
1 + 4 \xi^2 t^2, & |t| \le t_1,\\
\frac{1}{2t^2} + \frac12 + \xi (1-2|t|) + \frac52 \xi^2 t^2, & t_1 < |t| \le t_2,\\
\frac{1}{t^2} + \xi^2 t^2, & t_2 < |t| \le t_3,\\
2\xi, & t_3 < |t| \le t_4,\\
\frac{1}{2t^2} - \frac12 + \xi (1+2|t|) - \frac32 \xi^2 t^2, &  t_4 < |t| \le t_5,\\
0, & |t| > t_5.
\end{cases}
\end{equation*}
Here
\begin{align*}
t_1 &= t_1(\xi) = \frac{-1+\sqrt{1+4\xi}}{2\xi}, & t_2 & = t_2(\xi) = \frac{1-\sqrt{1-4\xi}}{2\xi}, & t_3 &= t_3(\xi) = \frac{1}{\sqrt{\xi}}, \\
t_4 &= t_4(\xi) = \frac{1+\sqrt{1-4\xi}}{2\xi}, & t_5 & = t_5(\xi) = \frac{1+\sqrt{1+4\xi}}{2\xi}.
\end{align*}
\end{theorem}

For the record, these values $t_i$ have the following asymptotics as $\xi\to 0$:
\begin{align*}
t_1 &= 1-\xi+O(\xi^2), &t_4 &= \xi^{-1}-1-\xi+O(\xi^2), \\ 
t_2 &= 1+\xi+O(\xi^2), &t_5 &= \xi^{-1}+1-\xi+O(\xi^2). 
\end{align*}

Note that in \cite{Koleda2016} the expression (12) for $\omega_2(\xi,t)$ contains a typo: in the first case $\xi$~must be squared. Here we use the correct version.

From the general formula \eqref{eq-phi-def} we have
$\phi_1(t)=\frac{1}{\max(1,t^2)}$.
From Theorem~\ref{thm-omega2} one can see, the difference $|\omega_2(\xi, t)-\phi_1(t)|$ has the order $O(\xi)$ in a neighborhood of $t=1$ as $\xi\to 0$.
Whereas when $n\ge 3$ Lemma \ref{lm-delta} (see Section \ref{sbsbsec-proof-idiff} below) gives the magnitude $O(\xi^2)$ for $|\omega_n(\xi, t)-\phi_{n-1}(t)|$ in the same neighborhood. For all other $t$, except the two intervals $|t|\in (t_1,t_2)$, the general bound from Lemma \ref{lm-delta} holds for the quadratic case too.

However, despite the uniform convergence of $\omega_2(\xi, t)$ to $\phi_1(t)$ as $\xi\to 0$,
the following theorem is true.
\begin{theorem}[\cite{Koleda2016}]
For $\xi\le 1/4$
\begin{equation*}
\int\limits_{-\infty}^{+\infty} \left(\omega_2(\xi, t) - \phi_1(t)\right) dt = 4 - \frac{16}{3} \sqrt{\xi} + O(\xi),
\end{equation*}
where the implicit constant in the big-O-notation is absolute.
\end{theorem}
As the reader will see from the arguments below, this appeared to be a general feature of real algebraic integers of any degree $n\ge 2$.

Concerning the total number of real quadratic algebraic integers, its asympotics can be obtained in a more precise form than in case $n\ge 3$.
\begin{theorem}[\cite{Koleda2016}]\label{thm-total-2}
For integer $Q\ge 4$
\begin{equation*}
\frac{\Omega_2(Q,\mathbb{R})}{4Q^2} = 2-\frac4{3\sqrt{Q}}
- \frac{\ln Q}{Q}  + \frac{2(1-\gamma)}{Q} + O\!\left(\frac1{Q\sqrt{Q}}\right),
\end{equation*}
where $\gamma = 0{,}5772\dots$ is Euler's constant; the implicit big-O-constant is absolute.
\end{theorem}

Note that
$2=2^{-2} \int_{\mathbb{R}}\phi_1(t)\,dt + 1$.

\subsection{Algebraic integers of higher degrees}\label{sec-ndeg}

Here we will estimate the difference $\omega_n(\xi, t)-\phi_{n-1}(t)$ for $n\ge 3$. As a result, we deduce Theorems \ref{thm-main} and \ref{thm-total}.

So, in this section, $n\ge 3$; and we assume, without loss of generality, $t\ge 0$ (since $\omega_n(\xi,t)$ and $\phi_{n-1}(t)$ are both even functions of $t$).

For the sake of simplification, introduce the following notation
\begin{align*}
\mathbf{p} &:= (p_1,\dots,p_{n-1}), \quad d\mathbf{p} := dp_1\,dp_2\dots dp_{n-1},\\
\mathbf{w}(t) &:= \left(t, t^2, \dots, t^{n-1}\right), \\
\mathbf{v}(t) &:= \left(1, 2t, \dots, (n-1)t^{n-2}\right) = \frac{d}{dt} \mathbf{w}(t).
\end{align*}

In this notation, the function $\phi_{n-1}(t)$ takes the form:
\[
\phi_{n-1}(t) = \int\limits_{G_{n-1}(t)} \left|\mathbf{v}(t)\cdot \mathbf{p}\right|\,d\mathbf{p}, \qquad t \in \mathbb{R},
\]
where
\begin{equation}\label{eq-G-def'}
G_{n-1}(t) = \left\{\mathbf{p} \in \mathbb{R}^{n-1} : \|\mathbf{p}\|_\infty \le 1, \ |\mathbf{w}(t)\cdot \mathbf{p}| \le 1 \right\}.
\end{equation}
Note that $G_{n-1}(t)$ differs from $D_n(\xi,t)$ in the absence of term $\xi t^n$ in the modulus brackets.

Change the variables in the integral for $\omega_n(\xi,t)$:
\[
\begin{cases}
p_i = q_i, & i=1,\dots,n-2,\\
p_{n-1} = q_{n-1} - \xi t.
\end{cases}
\]
Then the integral takes the form:
\[
\omega_n(\xi,t) = \int\limits_{S_n(\xi,t)} \left|\xi t^{n-1} + \sum_{k=1}^{n-1} k q_k t^{k-1}\right|\,dq_1\ldots\,dq_{n-1}, \qquad t \in \mathbb{R},
\]
where
\[
S_n(\xi,t) = \left\{(q_1,\ldots,q_{n-1}) \in \mathbb{R}^{n-1} :
\max\limits_{1\le i \le n-2}|q_i| \le 1, \
\begin{array}{l}
|q_{n-1}-\xi t| \le 1, \\
|q_{n-1} t^{n-1} + \ldots + q_1 t| \le 1
\end{array}
\right\}.
\]
Now, we partition the integral into the three summands
\begin{multline*}
\omega_n(\xi,t) = \int\limits_{G_{n-1}(t)} \left|\xi t^{n-1} + \mathbf{v}(t) \mathbf{q}\right|\,d\mathbf{q} \ +\\
+
\int\limits_{S_n^+(\xi,t)} \left|\xi t^{n-1} + \mathbf{v}(t) \mathbf{q}\right|\,d\mathbf{q} \  -
\int\limits_{S_n^-(\xi,t)} \left|\xi t^{n-1} + \mathbf{v}(t) \mathbf{q}\right|\,d\mathbf{q},
\end{multline*}
where $G_n(t)$ is defined in \eqref{eq-G-def}, and
\begin{align*}
S_n^+(\xi,t) &= \left\{ \mathbf{q} \in \mathbb{R}^{n-1} :
|\mathbf{w}(t) \mathbf{q}| \le 1, \
\max\limits_{1\le i \le n-2}|q_i| \le 1, \
\phantom{-}1 < q_{n-1} \le 1+\xi t
\right\},\\
S_n^-(\xi,t) &= \left\{ \mathbf{q} \in \mathbb{R}^{n-1} :
|\mathbf{w}(t) \mathbf{q}| \le 1, \
\max\limits_{1\le i \le n-2}|q_i| \le 1, \
-1 \le q_{n-1} < - 1+\xi t
\right\}.
\end{align*}
For convenience, denote the integral over $G_{n-1}(t)$ by $J_1$, that one over $S_n^+(\xi,t)$ by $J_2$, and that one over $S_n^-(\xi,t)$ by $J_3$. So
\[
\omega_n(\xi,t) = J_1 + J_2 - J_3.
\]

Now we will separately estimate $J_1-\phi_{n-1}(t)$ (the ``main'' part) and $J_2 - J_3$ (the ``corner'' part).

\subsubsection{Estimation of the main part of the difference}

Evaluating the difference $J_1 - \phi_{n-1}(t)$ is based on the following lemma.

\begin{lemma}\label{lm-diff}
Let $V\subset\mathbb{R}^k$ be a bounded region symmetric with respect to the origin of coordinates. Let $\mathbf{v} = (v_1,\dots,v_k)$ be a fixed nonzero vector, and let $\epsilon > 0$ be a real number.
Then
\[
\int\limits_V |\mathbf{v}\cdot\mathbf{x} + \epsilon|\,d\mathbf{x} -
 \int\limits_V |\mathbf{v}\cdot\mathbf{x}|\,d\mathbf{x} \ =
\int\limits_{V(\epsilon)} \left(\epsilon - |\mathbf{v}\cdot\mathbf{x}|\right)\,d\mathbf{x}.
\]
where
$V(\epsilon) := \left\{\mathbf{x}\in V : |\mathbf{v}\cdot\mathbf{x}| < \epsilon\right\}$.
\end{lemma}

Remark 1. Obviously, if $V(\epsilon)=V$, then
\begin{equation}\label{eq-V-int}
\int\limits_V |\mathbf{v}\cdot\mathbf{x} + \epsilon|\,d\mathbf{x} = \epsilon \cdot \mes_k V.
\end{equation}

Remark 2. It is easy to observe that for any $0<\lambda<1$
\[
(1-\lambda)\,\epsilon \cdot \mes_k V(\lambda\,\epsilon) \ \le \
 \int\limits_V |\mathbf{v}\cdot\mathbf{x} + \epsilon|\,d\mathbf{x} -
 \int\limits_V |\mathbf{v}\cdot\mathbf{x}|\,d\mathbf{x} \ \le \ \epsilon \cdot \mes_k V(\epsilon),
\]

\begin{proof}
Changing $\mathbf{x}$ for $-\mathbf{x}$ in the integral, after transformation, we obtain
\[
\int\limits_V |\mathbf{v}\cdot\mathbf{x} + \epsilon|\,d\mathbf{x} = \int\limits_V |\mathbf{v}\cdot\mathbf{x} - \epsilon|\,d\mathbf{x} =
\int\limits_V \frac{|\mathbf{v}\cdot\mathbf{x} + \epsilon|+|\mathbf{v}\cdot\mathbf{x} - \epsilon|}{2}\,d\mathbf{x}.
\]
Since
\[
\frac{|\mathbf{v}\cdot\mathbf{x} + \epsilon|+|\mathbf{v}\cdot\mathbf{x} - \epsilon|}{2} =
\begin{cases}
|\mathbf{v}\cdot\mathbf{x}|, & |\mathbf{v}\cdot\mathbf{x}| \ge \epsilon,\\
\epsilon, & |\mathbf{v}\cdot\mathbf{x}| < \epsilon,
\end{cases}
\]
the lemma follows immediately.
\end{proof}

From Lemma \ref{lm-diff}, we have
\begin{equation}\label{eq-J-phi-diff}
J_1 - \phi_{n-1}(t) = \int\limits_{U_n(\xi,t)} \left(\xi t^{n-1} - |\mathbf{v}(t) \mathbf{q}|\right) d\mathbf{q},
\end{equation}
where
\begin{equation}\label{eq-U-def}
U_n(\xi,t) = \left\{\mathbf{q}\in\mathbb{R}^{n-1} : \|\mathbf{q}\|_\infty \le 1, \ |\mathbf{w}(t) \mathbf{q}|\le 1, \ |\mathbf{v}(t) \mathbf{q}| \le \xi t^{n-1}\right\}.
\end{equation}
Note that $U_n(\xi,t)$ is a $(n-1)$--dimensional region defined by $n+1$ linear inequalities and, therefore, it is a $(n-1)$--dimensional polyhedron.

It is easy to check that if we take any $n-1$ inequalities in \eqref{eq-U-def}, then their left-hand sides are linearly independent as linear functions of $q_1,\dots,q_{n-1}$ for any real $t\ne 0$.

\begin{lemma}\label{lm-I1}
Let $n\ge 3$ be a fixed integer, and $0<\xi<1$.

A) For $t\in[t_1,t_2]$, where the parameters $t_1$ and $t_2$ have the asymptotics:
\[
t_1=2+\sqrt{2} + O(\xi), \qquad
t_2 = \xi^{-1/2} - 1 + O(\xi^{1/2}), 
\qquad \xi\to 0,
\]
with absolute big-O-constants,
$U_n(\xi,t)$ is an $(n-1)$--dimensional parallelepiped defined by:
\begin{equation}\label{eq-U2}
U_n(\xi,t) = \left\{\mathbf{q}\in\mathbb{R}^{n-1} : \max_{1\le k\le n-3} |q_k|\le 1, \ |\mathbf{w}(t) \mathbf{q}|\le 1, \ |\mathbf{v}(t) \mathbf{q}| \le \xi t^{n-1}\right\}.
\end{equation}

B) For all
\begin{equation}\label{eq-t-restr}
|t|\ge t_3 := \xi^{-1/2} + 1,
\end{equation}
$U_n(\xi,t)$ doesn't depend on $\xi$ and is a parallelepiped defined by:
\begin{equation}\label{eq-U3}
U_n(\xi,t) = G_{n-1}(t) = \left\{\mathbf{q}\in\mathbb{R}^{n-1} : \max_{1\le k\le n-2} |q_k|\le 1, \ |\mathbf{w}(t) \mathbf{q}|\le 1 \right\}.
\end{equation}
\end{lemma}

Remark. Note that $t_1$, $t_2$ and $t_3$ do not depend on $n$.

\begin{proof}
The overall idea of the proof is to find needless inequalities among the ones defining $U_n(\xi,t)$, that is, such inequalities that can be dropped from the definition of $U_n(\xi,t)$ without any effect.
Everywhere in the proof we assume $t>1$.

First of all, we rewrite the inequality $|\mathbf{w}(t) \mathbf{q}|\le 1$ in the form
\begin{equation}\label{eq-q-n-1}
q_{n-1} = - \sum_{i=2}^{n-1} \frac{q_{n-i}}{t^{i-1}} - \frac{\theta}{t^{n-1}}, \qquad |\theta|\le 1.
\end{equation}
Hence, we obviously have
\[
|q_{n-1}| \le \sum_{i=1}^{n-1} t^{-i} < \frac{1}{t-1}.
\]
Thus for any $t\ge 2$ the inequality $|q_{n-1}|\le 1$ follows from \eqref{eq-q-n-1}. Therefore, we can drop it.

At the next step, putting \eqref{eq-q-n-1} into the inequality $|\mathbf{v}(t) \mathbf{q}|\le \xi t^{n-1}$, we get after transformation:
\begin{equation}\label{eq-q-n-2}
\left|\sum_{i=2}^{n-1} \frac{(i-1)q_{n-i}}{t^{i-2}} + \frac{(n-1)\theta}{t^{n-2}}\right| \le \xi t^2.
\end{equation}
Now, the struggle will be between the inequalities \eqref{eq-q-n-2} and $|q_{n-2}|\le 1$: which one of them is ``stronger'' and will be allowed to stay.

\textbf{Part A.}
From \eqref{eq-q-n-2} we obtain
\[
|q_{n-2}| \le \xi t^2 + \sum_{k=2}^{n-1} k t^{-k+1} < \xi t^2 + \frac{t^2}{(t-1)^2} - 1.
\]
Therefore if $t$ satisfies the inequality
\begin{equation}\label{eq-t-U2}
\xi t^2 + \frac{t^2}{(t-1)^2} - 1 \le 1,
\end{equation}
then the inequality $|q_{n-2}|\le 1$ is needless.

After some meditation on \eqref{eq-t-U2} one can comprehend that for $t>1$ and sufficiently small $\xi$ the solutions of \eqref{eq-t-U2} form an interval $[t_1, t_2]$, where $t_1$ and $t_2$ are the solutions of
\begin{equation}\label{eq-t12}
\xi t^2 + t^2 (t-1)^{-2} = 2
\end{equation}
for $t>1$, and as $\xi\to 0$
\[
t_1 = t_1(\xi) = 2+\sqrt{2} + O(\xi),
\qquad t_2 = t_2(\xi) = \xi^{-1/2} - 1 + O(\xi^{1/2}), 
\]
with the absolute implied big-O-constants.
Let us explain this revelation in details. It is easy to check (using the derivative) that for $t>1$ the function $\xi t^2 + t^2 (t-1)^{-2}$ has the only minimum at $t=\xi^{-1/3}+1$ (everywhere else the derivative has non-zero values, and therefore the function is monotonic), and this minimum is equal to $\xi^{1/3} (1+\xi^{1/3})^2$, which is less than 2 for small $\xi$.
Hence, one has $t_1< \xi^{-1/3}+1<t_2$. Thus, obviously,
\[
\lim_{\xi\to 0} \xi t_1^2 = 0, \qquad \lim_{\xi\to 0} t_2 = +\infty.
\]
Placing this into \eqref{eq-t12} shows that
\[
\frac{t_1(0)}{t_1(0)-1} = \sqrt{2}, \qquad \lim_{\xi\to 0} \xi t_2^2 = 1,
\]
where $t_1(0) := \lim_{\xi\to 0} t_1$. So from \eqref{eq-t12} one immediately has
\[
t_1^2 (t_1-1)^{-2} = 2 + O(\xi), \qquad \xi t_2^2 = 1 + O(\xi^{1/2}).
\]
Note that $t_1$ monotonously decreases, and $t_2$ monotonously increases as $\xi\to 0$.

One can recover the development of $t_2$ by degrees of $\tilde{\xi} := \xi^{1/2}$ putting the expression with undetermined coefficients
\[
t_2 = \tilde\xi^{-1} + c_0 + c_1 \tilde\xi + \dots
\]
into \eqref{eq-t12} and solving this with respect to the unknown coefficients $c_i$.

Note that $t_2$ is an analytic (and algebraic) function of $\tilde\xi$ in a punctured neighbourhood of $\tilde\xi=0$ and has a simple pole at $\tilde\xi=0$. Therefore the aforementioned development does exist.
At the same time, $t_1$ is an analytic function of $\xi$ in a neighbourhood of $\xi=0$. Hence, it is possible to find a development of $t_1$ in the form $t_1 = d_0 + d_1 \xi + d_2 \xi^2 + \dots$.
Notice that $t_2$ is represented by a series in terms of $\xi^{1/2}$ rather than $\xi$, in contrast to $t_1$.

So Part A is proved now, and from \eqref{eq-U-def} we have \eqref{eq-U2}.

\textbf{Part B.}
Now, since $|q_k|\le 1$, for $1\le k\le n-1$, and $|\theta|\le 1$, the left-hand side of \eqref{eq-q-n-2} can be estimated as:
\[
\left|\sum_{i=2}^{n-1} \frac{(i-1)q_{n-i}}{t^{i-2}} + \frac{(n-1)\theta}{t^{n-2}}\right| \le
\sum_{k=1}^{n-1} k t^{-k+1} < \frac{t^2}{(t-1)^2}.
\]
Thus, obviously, if $t$ satisfies the inequality
\begin{equation}\label{eq-t3-ineq}
\frac{t^2}{(t-1)^2} \le \xi t^2,
\end{equation}
then \eqref{eq-q-n-2} surely loses to $|q_{n-2}|\le 1$, and so $|\mathbf{v}(t)\cdot \mathbf{q}|\le \xi t^{n-1}$ can be excluded without any effect.
And in fact, we have \eqref{eq-U3} (cf. \eqref{eq-G-def}, one can nominally append the inequality $|q_{n-1}|\le 1$).
Solving \eqref{eq-t3-ineq} we find that this happens when $t\ge \xi^{-1/2} + 1$.
\end{proof}

For subsequent calculations we need this lemma.
\begin{lemma}\label{lm-integr}
Let $A=(a_{ij})_{i,j=1}^d$ be an invertible real matrix, and let $\mathbf{a}_i = (a_{i1},\dots,a_{id})$ be its $i$-th row.
Let $V\subset\mathbb{R}^d$ be the set of all $\mathbf{x}=(x_1,\dots,x_d)^T\in\mathbb{R}^d$ satisfying the system of $d$~linear inequalities
\begin{equation}\label{eq-parall}
|\mathbf{a}_i\cdot \mathbf{x}|\le h_i, \qquad 1\le i\le d.
\end{equation}
Then for any integrable function $f:\mathbb{R}^d\to\mathbb{R}$
\[
\int_V f(\mathbf{x})\,d\mathbf{x} = \frac{1}{|\det A|} \int_{-h_d}^{h_d}\dots\int_{-h_2}^{h_2}\int_{-h_1}^{h_1} f(A^{-1}\mathbf{t}) dt_1 dt_2 \dots dt_d.
\]
\end{lemma}

Remark. We are intrested in the two following special cases:
\newline
a) when $f(\mathbf{x})=1$, then
\begin{equation}\label{eq-vol-par}
\mes_d V = \int_V d\mathbf{x} = \frac{2^d}{|\det A|}\, h_1 h_2 \dots h_d;
\end{equation}
b) when $f(\mathbf{x})=h_i-|\mathbf{a}_i\cdot \mathbf{x}|$, then
\begin{equation}\label{eq-lin-par}
\int_V (h_i - |\mathbf{a}_i\cdot \mathbf{x}|)\,d\mathbf{x} = \frac{2^{d-1} h_i}{|\det A|}\,h_1 h_2 \dots h_d.
\end{equation}

\begin{proof}
To calculate the integral, we make the change of variables:
$\mathbf{t} = A\mathbf{x}$.
The Jacobian is equal to
\[
\det \frac{\partial \mathbf{x}}{\partial \mathbf{t}} = \left(\det\frac{\partial \mathbf{t}}{\partial \mathbf{x}}\right)^{-1} = \left(\det A \right)^{-1}.
\]
The system \eqref{eq-parall} takes the form: $|t_i|\le h_i$ for $1\le i\le d$.
Hence, the lemma follows.
\end{proof}

\begin{lemma}\label{lm-J1-phi-diff}
Let $n\ge 3$ be a fixed integer, and $0<\xi<1$.

A) For $|t| \in [0, 1]$
\begin{equation*}\label{eq-est-1}
0\le J_1 - \phi_{n-1}(t) \le 2^{n-2} \xi^2 t^{2n-2}.
\end{equation*}

B) For $|t|\in[1, \xi^{-1/2}+1]$, the following inequality holds
\begin{equation*}\label{eq-est-2}
0< J_1 - \phi_{n-1}(t) \le 2^{n-2} \xi^2 t^2.
\end{equation*}
Moreover, if $|t|\in [t_1, t_2]$, where $t_1$ and $t_2$ are the same as in Lemma \ref{lm-I1}, then we have the equality
\[
J_1 - \phi_{n-1}(t) = 2^{n-2} \xi^2 t^2.
\]

C) For $|t|\ge \xi^{-1/2} + 1$, we have
\begin{equation*}
J_1 = 2^{n-1} \xi.
\end{equation*}
\end{lemma}
\begin{proof}
The idea of the proof is to replace $U_{n-1}(\xi,t)$ by a parallelepiped, which includes $U_{n-1}(\xi,t)$, and then apply Lemma \ref{lm-integr}. Since $U_{n-1}(\xi,t)$ is an $(n-1)$--dimensional region, the easiest way to do this is to discard two of the $n+1$ linear inequalities defining $U_{n-1}(\xi,t)$.

\textbf{Part A}. To prove this part we discard the two inequalities: $|q_1|\le 1$ and $|\mathbf{w}(t)\cdot \mathbf{q}|\le 1$. Now we have
\[
U_{n}(\xi,t) \subseteq W_1 :=
\left\{\mathbf{q}\in\mathbb{R}^{n-1} : \max_{2\le k\le n-1} |q_k|\le 1, \ |\mathbf{v}(t) \mathbf{q}| \le \xi t^{n-1}\right\}.
\]
Now to estimate the difference $J_1 - \phi_{n-1}(t)$ from above, we replace $U_n(\xi,t)$ by $W_1$ in \eqref{eq-J-phi-diff}:
\[
J_1 - \phi_{n-1}(t) \le \int\limits_{W_1} (\xi t^{n-1} - |\mathbf{v}(t) \mathbf{q}|)\, d\mathbf{q}.
\]
Now we obtain the estimate of Part A, applying \eqref{eq-lin-par} with
\[
d = n-1, \quad
A = \left(\begin{matrix}
1 & \dots\\
0 & E_{n-2}
\end{matrix}\right),
\quad h_1 = \xi t^{n-1}, \quad h_i = 1, \quad 2\le i\le n-1,
\]
where $E_k$ is the identity matrix of size $k$.

\medskip

\textbf{Part B.}
Now we eliminate the restrictions $|q_{n-1}|\le 1$ and $|q_{n-2}|\le 1$ and obtain
\[
U_{n}(\xi,t) \subseteq W_2 :=
\left\{\mathbf{q}\in\mathbb{R}^{n-1} : \max_{1\le k\le n-3} |q_k|\le 1, \ |\mathbf{w}(t) \mathbf{q}|\le 1, \ |\mathbf{v}(t) \mathbf{q}| \le \xi t^{n-1}\right\}.
\]
Note that by Lemma \ref{lm-I1} we have $U_{n}(\xi,t) = W_2$ for $t_1\le t\le t_2$.

Just like in Part A, we estimate
\[
J_1 - \phi_{n-1}(t) \le \int\limits_{W_2} (\xi t^{n-1} - |\mathbf{v}(t) \mathbf{q}|)\, d\mathbf{q}.
\]
using \eqref{eq-lin-par} with $d=n-1$, $h_i=1$ for $1\le i\le n-2$, $h_{n-1}=\xi t^{n-1}$, and
\[
A = \left(\begin{matrix}
E_{n-3} & 0 & 0\\
\dots & t^{n-2} & t^{n-1}\\
\dots & (n-2)t^{n-3} & (n-1) t^{n-2}
\end{matrix}\right).
\]
Note that for $n=3$ this is just a $2\times 2$-matrix; the calculation still holds true.

\medskip

\textbf{Part C.}
From Lemma \ref{lm-I1}, applying \eqref{eq-V-int} to $J_1$ with $V = G_{n-1}(t)$, we have
\[
J_1 = \xi t^{n-1} \mes_{n-1} G_{n-1}(t).
\]
Using \eqref{eq-vol-par} with $d=n-1$, $h_i=1$ for $1\le i\le n-1$, and
\[
A = \left(\begin{matrix}
E_{n-2} & 0\\
\dots & t^{n-1}
\end{matrix}\right),
\]
one finds $\mes_{n-1} G_{n-1}(t) = 2^{n-1} t^{-n+1}$.
The lemma is proved.
\end{proof}

\subsubsection{Estimation of the ``corner'' part}

The aim of this section is to prove the following statement.
\begin{lemma}\label{lm-J2-J3-diff}
For sufficiently small $\xi>0$, there exist computable values $\widetilde{t}_1$, $\widetilde{t}_2$ and $\widetilde{t}_3$ independent of $n$ and satisfying the asymptotics
\[
\widetilde{t}_1 = 2 + 2\xi + O(\xi), \qquad \widetilde{t}_2 = \xi^{-1} - 1 + O(\xi), \qquad \widetilde{t}_3 = \xi^{-1} + 1,
\]
such that the following holds:

a) for $t\le \widetilde{t}_1$
\[
|J_2 - J_3| \ll_n \xi^2 t^n;
\]

b) for $\widetilde{t}_1 \le t < \widetilde{t}_2$
\[
J_2 = J_3 = 0;
\]

c) for $t\ge \widetilde{t}_2$
\[
J_2 = 0, \quad 0\le J_3 \le J_1;
\]

d) for $t>\widetilde{t}_3$, we have
$J_1 = J_3$ and so $\omega_n(\xi,t) = 0$.
\end{lemma}

\begin{proof}
First of all, we find for which $t$ the equality $S_n(\xi,t)=G_{n-1}(t)$ holds.

Any $(q_1, q_2,\dots, q_{n-1})\in S_n(\xi,t)$ satisfies the system
\begin{equation}\label{eq-S-ineq}
\begin{cases}
-1+\xi t \le q_{n-1} \le 1+\xi t,\\
\frac{-1-q_1 t - \dots - q_{n-2} t^{n-2}}{t^{n-1}} \le q_{n-1} \le \frac{1-q_1 t - \dots - q_{n-2} t^{n-2}}{t^{n-1}}.
\end{cases}
\end{equation}
Obviously
\[
\min\limits_{\substack{|q_i|\le 1\\1\le i\le n-2}} \frac{-1-q_1 t - \dots - q_{n-2} t^{n-2}}{t^{n-1}} = -\sum_{k=1}^{n-1} t^{-k}.
\]
Therefore, the first-line inequality in \eqref{eq-S-ineq}
is needless (and we can just omit it) if $t$ satisfies the inequality
\begin{equation}\label{eq-t-cond}
-\sum_{k=1}^{n-1} t^{-k} \ge -1+\xi t.
\end{equation}
Apparently, for such $t$ we have $S_n(\xi,t)=G_{n-1}(t)$.
So our aim now is to find explicitly values $t$ for which $S_n(\xi,t)=G_{n-1}(t)$ for sure.

To give uniform bounds on $t$, we need to simplify \eqref{eq-t-cond} in some uniform (independent of $n$) way. It is easy to check that any solution $t$ of \eqref{eq-t-cond} must satisfy the condition $t>1$; otherwise, we would have a contradiction.
The inequality \eqref{eq-t-cond} is equivalent to
\[
-\frac{1}{t-1}\left(1-\frac{1}{t^{n-1}}\right) \ge -1+\xi t.
\]
Hence, one can readily see that any solution $t>1$ of
\begin{equation*}
-\frac{1}{t-1} \ge -1+\xi t
\end{equation*}
also meets \eqref{eq-t-cond}.
Solving the latter inequality, we get that $S_n(\xi,t)=G_{n-1}(t)$ for
\[
\widetilde{t}_1 := \frac{1+\xi - \sqrt{(1+\xi)^2-8\xi}}{2\xi} \le t \le \frac{1+\xi + \sqrt{(1+\xi)^2-8\xi}}{2\xi} =: \widetilde{t}_2,
\]
where $\widetilde{t}_1$ and $\widetilde{t}_2$ are the roots of the equation:
\begin{equation}\label{eq-t-cond-bnd}
\xi t^2 - (\xi+1) t + 2 = 0.
\end{equation}
One can directly develop $\widetilde{t}_1$ and $\widetilde{t}_2$ by degrees of $\xi$ and obtain the first few terms using undetermined coefficients with \eqref{eq-t-cond-bnd}:
\[
\widetilde{t}_1 = 2 + 2 \xi + O(\xi^2), 
\qquad \widetilde{t}_2 = \xi^{-1} - 1 + O(\xi). 
\]

Consider separately the three intervals: $[0,\widetilde{t}_1)$, $[\widetilde{t}_1,\widetilde{t}_2)$, $[\widetilde{t}_2, +\infty)$.

a) Let $0\le t < \widetilde{t}_1$.

In the integral $J_2$, make the change $q_{n-1}=1+\theta$. In the integral $J_3$, we make the change $\mathbf{q} \to -\mathbf{q}$ and next the change $q_{n-1} = 1-\theta$. Here $0<\theta<\xi t$.
Thus, we obtain
\begin{align*}
J_2 &= \int\limits_{\widetilde{S}_n^+(\xi,t)} \left|\xi t^{n-1} + (n-1)(1+\theta)t^{n-2} + \sum_{k=1}^{n-2} k q_k t^{k-1} \right|dq_1\,dq_2 \dots dq_{n-2}\,d\theta,\\
J_3 &= \int\limits_{\widetilde{S}_n^-(\xi,t)} \left|-\xi t^{n-1} + (n-1)(1-\theta)t^{n-2} + \sum_{k=1}^{n-2} k q_k t^{k-1} \right|dq_1\,dq_2 \dots dq_{n-2}\,d\theta,
\end{align*}
where
\begin{align*}
\widetilde{S}_n^+(\xi,t) &= \left\{ (q_1,\dots,q_{n-2},\theta) \in [-1,1]^{n-2}\times[0,\xi t] \ : \
\left|\sum_{k=1}^{n-2} q_k t^k + (1+\theta)t^{n-1}\right| \le 1
\right\},\\
\widetilde{S}_n^-(\xi,t) &= \left\{ (q_1,\dots,q_{n-2},\theta) \in [-1,1]^{n-2}\times[0,\xi t] \ : \
\left|\sum_{k=1}^{n-2} q_k t^k + (1-\theta)t^{n-1}\right| \le 1
\right\}.
\end{align*}
Denote the intersection and the union of $\widetilde{S}_n^-(\xi,t)$ and $\widetilde{S}_n^+(\xi,t)$:
\[
S_* = \widetilde{S}_n^-(\xi,t) \cap \widetilde{S}_n^+(\xi,t), \qquad
S^* = \widetilde{S}_n^-(\xi,t) \cup \widetilde{S}_n^+(\xi,t).
\]
Then we can write
\begin{align*}
|J_2 - J_3|\ &\le \ 2n\xi t^{n-1} \mes_{n-1} S_* + \left(n\xi t^{n-1} + \sum_{k=1}^{n-1} k t^{k-1}\right) \mes_{n-1} (S^*\setminus S_*)\\
& \le \ 2n\xi t^{n-1} \mes_{n-1} S^* + \left(\sum_{k=1}^{n-1} k t^{k-1}\right) \mes_{n-1} (S^*\setminus S_*).
\end{align*}
Obviously, $\mes_{n-1} S^*\le 2^{n-2} \xi t$.
The difference $S^*\setminus S_*$ may be represented in the form
\[
S^*\setminus S_* =
\left\{ (q_1,\dots,q_{n-2},\theta) \in [-1,1]^{n-2}\times[0,\xi t] \ : \
\frac{1}{t} - \theta t^{n-2} \le \left|q_1 +
\dots
\right| \le \frac{1}{t} + \theta t^{n-2}
\right\}.
\]
Hence,
\[
\mes_{n-1} (S^*\setminus S_*) \le 2^{n-1} t^{n-2} \int_0^{\xi t} \theta\,d\theta =
2^{n-2} \xi^2 t^n.
\]
Finally, for $t\in [0, \widetilde{t}_1)$ we have
\[
|J_2 - J_3| \le 2^{n-1} n \xi^2 t^n + \left(\sum_{k=1}^{n-1} k t^{k-1}\right) 2^{n-2} \xi^2 t^n =
2^{n-2} \left(2n + \sum_{k=1}^{n-1} k t^{k-1}\right) \xi^2 t^n.
\]

b) For $t\in [\widetilde{t}_1,\widetilde{t}_2)$ we have $S_n(\xi,t) = G_{n-1}(t)$, and so
\[
J_2 = J_3 = 0.
\]

c) For $t\ge \widetilde{t}_2$, since $S_n^+(\xi,t) = \varnothing$ and $S_n^-(\xi,t) \subseteq G_{n-1}(t)$, we have
\[
J_2 = 0, \quad 0\le J_3 \le J_1.
\]

d) Looking at the system \eqref{eq-S-ineq}, one can observe that $J_1 = J_3$ and so $\omega_n(\xi,t) = 0$, if
\begin{equation}\label{eq-t-cond2}
\sum_{k=1}^{n-1} t^{-k} \le -1 + \xi t.
\end{equation}
Now we simplify this inequality in the same way as \eqref{eq-t-cond}.
Evidently, \eqref{eq-t-cond2} has no positive solutions $t\le1$,
and is equivalent to
\[
\frac{t}{t-1}\left(1-\frac{1}{t^n}\right) \le \xi t.
\]
Hence, any $t>1$ satisfying
\begin{equation}\label{eq-t-cond2-x}
\frac{1}{t-1} \le \xi
\end{equation}
obeys \eqref{eq-t-cond2} too.
So, for these $t$, we surely have $J_1 = J_3$ and hence $\omega_n(\xi,t) = 0$.

The lemma is proved.
\end{proof}

\subsubsection{Proof of Theorems \ref{thm-main} and \ref{thm-total}. Final steps}\label{sbsbsec-proof-idiff}

Lemmas \ref{lm-J1-phi-diff} and \ref{lm-J2-J3-diff} immediately imply the corrollary.
\begin{lemma}\label{lm-delta}
Let $n\ge 3$ be a fixed integer. For all sufficiently small $\xi>0$, the function $\delta_n(\xi,t) := \omega_n(\xi,t)-\phi_{n-1}(t)$ can be estimated as follows.

A) From above we have
\[
\delta_n(\xi,t)\le
\begin{cases}
c(n) \xi^2 t^n, & |t|\le \widetilde{t}_1, \\
2^{n-2} \xi^2 t^2, & \widetilde{t}_1 < |t| < \xi^{-1/2}+1, \\
2^{n-1} \xi - \phi_{n-1}(t), & \xi^{-1/2}+1 \le |t| < \xi^{-1}+1,\\
-\phi_{n-1}(t), & |t|\ge \xi^{-1}+1,
\end{cases}
\]
where $c(n)$ is an explicit constant depending on $n$ only.

B) From below:
\[
\delta_n(\xi,t)\ge
\begin{cases}
\phantom{-}0, & |t|\le \widetilde{t}_2,\\
-\phi_{n-1}(t), & |t| > \widetilde{t}_2.
\end{cases}
\]

C) Exact value:
\[
\delta_n(\xi,t) =
\begin{cases}
2^{n-2} \xi^2 t^2, & t_1 \le |t| \le t_2, \\
2^{n-1} \xi - \phi_{n-1}(t), & \xi^{-1/2}+1 \le |t| < \widetilde{t}_2,\\
-\phi_{n-1}(t), & |t|\ge \xi^{-1}+1.
\end{cases}
\]

Here $t_1$, $t_2$ (both defined in Lemma \ref{lm-I1}), and $\widetilde{t}_1$, $\widetilde{t}_2$ (both defined in Lemma \ref{lm-J2-J3-diff}) have the asymptotics:
\begin{align*}
t_1 &= 2+\sqrt{2}+O(\xi), & t_2 &=\xi^{-1/2}-1+O(\xi^{1/2}),\\
\widetilde{t}_1 &= 2+O(\xi), &\widetilde{t}_2 &=\xi^{-1}-1+O(\xi).
\end{align*}

\end{lemma}

In \eqref{eq-omega} (see Theorem \ref{thm-main}) the enpoints $t=\pm \xi^{-1/2}$ and $t=\pm (\xi^{-1}+1)$ are obtained by adjusting $t=\xi^{-1/2}+1$ and $t=\widetilde{t}_2$ respectively. It is possible because we can hide in the remainder term (just adjusting its big-O-constant) anything that is $O(\xi)$.

For some values $t$, the function $\delta_n(\xi,t)$ is tough to calculate.
Fortunately, the remainder term in~\eqref{eq-O/omega} overrides some of these complexities.
Thus, to hide all ``roughnesses'' of $\delta_n(\xi,t)$, we invent a new function $\widetilde\delta_n(\xi,t)$ that assembles essential properties of $\delta_n(\xi,t)$ in a simple form.
Keeping
\[
\int_{-\infty}^{\infty} \left|\delta_n(Q^{-1},t) - \widetilde\delta_n(Q^{-1},t)\right| = O(Q^{-1}).
\]
is enough to get \eqref{eq-main} from \eqref{eq-O/omega}.

The following lemma shows how to build such a function $\widetilde\delta_n(\xi,t)$.
\begin{lemma}\label{lm-delta2}
Consider the function
\[
\widetilde\delta_n(\xi,t) =
\begin{cases}
2^{n-2} \xi^2 t^2, & |t| \le \xi^{-1/2}, \\
2^{n-1} \xi - \phi_{n-1}(t), & \xi^{-1/2} < |t| \le \xi^{-1},\\
-\phi_{n-1}(t), & |t| > \xi^{-1}.
\end{cases}
\]
Then
\[
\int_{-\infty}^{\infty} \left|\delta_n(\xi,t) - \widetilde\delta_n(\xi,t)\right| dt = O(\xi),
\]
where the big-O-constant depends on $n$ only.
\end{lemma}
\begin{proof}
Here we use the notation of Lemma \ref{lm-delta}. For the sake of brevity, in this proof we omit the expression $|\delta_n(\xi,t) - \widetilde\delta_n(\xi,t)| dt$ under integrals.

For values $t$ in the intervals $[0,t_1)$, $[t_2,\xi^{-1/2}+1)$ and $\left[\widetilde{t}_2,\xi^{-1}+1\right]$, we trivially estimate
\[
|\delta_n(\xi,t) - \widetilde\delta_n(\xi,t)| \le |\delta_n(\xi,t)|+|\widetilde\delta_n(\xi,t)| \ll_n
\begin{cases}
\xi^2, & |t|\le t_1,\\
\xi, & |t|\ge t_2.
\end{cases}
\]
Since the length of all the three intervals is $O(1)$, we readily obtain
\[
\int_0^{t_1} \ll_n \xi^2, \qquad
\int_{t_2}^{\xi^{-1/2}+1} \ll_n \xi, \qquad
\int_{\widetilde{t}_2}^{\xi^{-1}+1} \ll_n \xi.
\]

For the rest of values $t$ we obviously have $\delta_n(\xi,t) = \widetilde\delta_n(\xi,t)$, and thus
\[
\int_{t_1}^{t_2} = 0,
\qquad
\int_{\xi^{-1/2}+1}^{\widetilde{t}_2} = 0,
\qquad
\int_{\xi^{-1}+1}^\infty = 0.
\]

Noticing that both $\delta_n(\xi,t)$ and $\widetilde\delta_n(\xi,t)$ are even finishes the proof of the lemma.
By the way, we proved Theorem~\ref{thm-main} too, where $\widetilde\omega_n(\xi,t):= \phi_{n-1}(t)+\widetilde\delta_n(\xi,t)$.
\end{proof}

Now we prove Theorem \ref{thm-total}.
From Lemma \ref{lm-delta2} we have
\[
\int_0^{\infty} (\omega_n(\xi,t)-\phi_{n-1}(t))\,dt =
\int_0^{\infty} \widetilde\delta_n(\xi,t)\,dt + O(\xi).
\]
Joining the integrals over the intervals $[0,\xi^{-1/2})$, $[\xi^{-1/2}, \xi^{-1})$, and $[\xi^{-1},+\infty)$,
we have
\[
\begin{aligned}
\int_0^{\infty} \widetilde\delta_n(\xi,t)\,dt &=
2^{n-2}\left(
\xi^2 \int_0^{\xi^{-1/2}} t^2\, dt +
2\xi \int_{\xi^{-1/2}}^{\xi^{-1}} dt\right) -
\int_{\xi^{-1/2}}^\infty \phi_{n-1}(t)\,dt \\
&=2^{n-1} \left(1 - \frac{4}{3}\xi^{1/2}\right) + O(\xi^{3/2}),
\end{aligned}
\]
where the integral
\[
\int_{\xi^{-1/2}}^\infty \phi_{n-1}(t)\,dt = 2^{n-2}\xi^{1/2} + O(\xi^{3/2})
\]
is calculated for $\xi\le (2+\sqrt{2})^{-2}$ using the expression~\eqref{eq-analit-f2}.

Since both $\omega_n(\xi,t)$ and $\phi_{n-1}(t)$ are even functions with respect to~$t$, we get
\begin{equation}\label{eq-idiff}
2^{-n}\int_{-\infty}^\infty \left(\omega_n(\xi,t)-\phi_{n-1}(t)\right)dt = 1 - \frac{4}{3}\,\xi^{1/2} + O(\xi).
\end{equation}
The latter equation and Theorem \ref{thm-int-counterpart} together give Theorem \ref{thm-total} for $n\ge 3$.
The quadratic case is covered by Theorem~\ref{thm-total-2}.

\addcontentsline{toc}{section}{References}

\bibliographystyle{abbrv}
\bibliography{bib8-eng}

\bigskip

{\small\noindent Denis Koleda (Dzianis Kaliada)}\\
{\footnotesize
{Institute of Mathematics, National Academy of Sciences of Belarus,\\
220072 Minsk, Belarus}\\
e-mail: koledad@rambler.ru
}

\end{document}